\documentclass{amsart}
\usepackage{amsfonts}
\usepackage{amsmath,amssymb}
\usepackage{amsthm}
\usepackage{amscd}
\usepackage{graphics}
\usepackage{graphicx}
{
\theoremstyle{definition}
\newtheorem{Def}{Definition}
\newtheorem{Ex}{Example}

}
\newtheorem{Cor}{Corollary}
\newtheorem{Prop}{Proposition}
\newtheorem{Thm}{Theorem}

\begin{document}
\title{Constructions of round fold maps on circle bundles}
\author{Naoki Kitazawa}
\address{Department of Mathematics, Tokyo Institute of Technology, 
2-12-1 Ookayama, Meguro-ku, Tokyo 152-8551, JAPAN, Tel +81-(0)3-5734-2205, 
Fax +81-(0)3-5734-2738,
}
\email{kitazawa.n.aa@m.titech.ac.jp}
\maketitle
\begin{abstract}
{\it {\rm (}Stable{\rm )} fold maps} are fundamental tools in a generalization of the theory of Morse functions on smooth manifolds
 and its application to studies of algebraic and differentail topological properties of smooth manifolds. {\it Round} fold maps were introduced as stable fold
 maps with {\it singular value sets}, or the set consisting of all the singular values, of concentric spheres by the author
 in 2013; for example, some {\it special generic} maps on spheres
 are regarded as round fold maps whose singular value sets are connected. \\
\ \ \ To construct explicit fold maps has been a fundamental and difficult problem. 
 In 2014, the author succeeded in constructing explicit round fold maps into a Euclidean
 space of dimension larger than $1$ on bundles
 over the standard sphere whose dimension is equal to that of the Euclidean space and connected
 sums of bundles over the standard sphere with fibers diffeomorphic to standard spheres by constructing maps locally
 and gluing them together properly. Later, the author constructed such maps
 on bundles
 over spheres or more general manifolds including families of bundles whose fibers are circles over given manifolds by applying operations
 compatiable with the structures of bundles ({\it P-operations}). \\
\ \ \ In this paper, we obtain new round fold maps on closed smooth manifolds having the structures of bundles whose fibers are
 circles over given closed manifolds by applying P-operations under weaker condtions. We also study algebraic and differential topological properties of the
 resulting maps and the source manifolds.
\end{abstract}
\section{Introduction and terminologies}
\label{sec:1}
{\it Fold maps} are important in generalizing the theory of Morse functions and its application to studies of geometry of manifolds. Studies of
 such maps were started by Whitney (\cite{whitney}) and Thom (\cite{thom}) in the 1950's. A {\it fold map} is a smooth map
 whose singular points are of the form
$$(x_1, \cdots, x_m) \mapsto (x_1,\cdots,x_{n-1},\sum_{k=n}^{m-i}{x_k}^2-\sum_{k=m-i+1}^{m}{x_k}^2)$$ for two positive integers $m \geq n$ and an integer $0 \leq i \leq m-n+1$ (for each singular point, we can determine
 the integer $i$ uniquely and the integer is said to be the {\it index} of the point). A Morse function
 is regarded as a fold map (n=1).
  For a fold map from a closed smooth manifold of dimension $m$ into a smooth manifold of dimension $n \geq 1$
 (without boundary), the followings hold.
\begin{enumerate}
\item The {\it singular set}, which is defined as the set of all the singular points, and the set of all the singular points whose indices are $i$ are closed smooth submanifolds of dimension $n-1$ of the source manifold. 
\item The restriction map to the singular set is a smooth immersion of codimension $1$.
\end{enumerate}
We also note that if the restriction map to the singular set is an immersion with normal
 crossings, then it is {\it stable}. Stable maps are important in the theory of global singularity; see \cite{golubitsky guillemin} for example). \\
\ \ \ Since around the 1990s, some classes of fold maps easy to handle have been
 actively studied. For example, in \cite{burlet derham}, \cite{furuya porto}, \cite{saeki 2}, \cite{saeki sakuma} and \cite{sakuma}, {\it special
 generic} maps were studied. A {\it special generic} map is defined as a fold map such that the indices of singular points are always $0$. From Morse functions and their
 singular points (singular sets), we can know some topological properties such as Euler numbers and homology groups and from special generic
 maps and their singular points (singular sets), we can often know more precise information of their source manifolds such as
 the homotopy types and the homeomorphism and diffeomorphism types. For
 example, homotopy spheres are characterized as manifolds admitting Morse functions with two singular points except the case where the sphere is $4$-dimensional and special generic
 maps often restrict the diffemorphism types of homotopy spheres of source manifolds (Standard spheres always admit special generic maps into Euclidean spaces whose dimensions are not higher than those of the source manifolds). As another example, 
the diffeomorphism types of manifolds admitting special generic maps were determined under appropriate conditions. For such facts, see \cite{saeki 2} and \cite{saeki sakuma} for example. In \cite{sakuma}, Sakuma studied
 {\it simple} fold maps. A simple fold map is defined as a fold map such that the inverse image of each singular value does not have any connected component with more
 than one singular points (see also \cite{saeki}) and  special generic maps are simple, for example. In
 \cite{kobayashi saeki}, Kobayashi and Saeki
 investigated topology of {\it stable} maps including fold maps which are stable into the plane. In \cite{saeki suzuoka}, Saeki and Suzuoka
 found good properties of manifolds admitting stable maps such that the inverse images of regular values are always disjoint unions of spheres. \\
\ \ \ Later, in \cite{kitazawa 2}, {\it round} fold maps, which will be mainly studied in this paper, were introduced. A {\it round} fold map is defined as a fold
 map satisfying the followings.
\begin{enumerate}
\item The singular set is a disjoint union of standard spheres.
\item The restriction map to the singular set is an embedding.
\item The {\it singular value set}, which is defined as the set of all the singular values of the map, is a disjoint union of spheres embedded concentrically. 
\end{enumerate}

The $m$-dimensional standard
 sphere $S^m$ ($m \geq 2$) admits a round fold map whose singular set is connected into ${\mathbb{R}}^n$ for $2 \leq n \leq m$. Such
 a map is also a special generic map and some special generic maps on homotopy spheres are round fold maps whose singular
 sets are connected as discussed in \cite{saeki 2}.  \\
\ \ \ In \cite{kitazawa 2}, homology groups and homotopy groups of manifolds admitting round fold maps are studied. Some examples of round fold maps
 and the diffeomorphism types of their source manifolds are given by the author in \cite{kitazawa}, \cite{kitazawa 2}, \cite{kitazawa 4} and \cite{kitazawa 5}. For example, in the first two papers, we have
 obtained round fold maps on $m$-dimensional closed smooth manifolds having the structures of bundles over the $n$-dimensional ($n \geq 2$) standard sphere $S^n$ whose
 fibers are closed smooth manifolds and whose structure groups consist of diffeomorphisms. It has been
 a very difficult problem to construct explicit fold maps, although the existence problems for fold maps have been solved under a lot of situations (see \cite{eliashberg} and \cite
{eliashberg 2} for example). The author has given a lot of
 answers for the construction problems, which will make it easier to study topological properties of a lot of smooth manifolds
 by using fold maps being not so complex. \\
\ \ \ Here, we introduce terminologies on (fiber) bundles, which are important in this paper. In this paper, a {\it bundle} means
 a bundle whose fibers are diffeomorphic
 closed smooth manifolds and whose structure group consists of diffeomorphisms on the fibers. We call a bundle whose fiber is a closed
 smooth manifold $X$ an {\it $X$-bundle}. \\
\ \ \ We return to precise presentations about explicit round fold maps. In \cite{kitazawa 2} and \cite{kitazawa 4}, the author
 constructed round fold maps into ${\mathbb{R}}^n$ ($n \geq 2$) also on $m$-dimensional ($m>n$) closed
 smooth manifolds represented as connected
 sums of bundles over the $n$-dimensional standard sphere $S^n$ with fibers diffemorphic
 to the ($m-n$)-dimensional standard sphere $S^{m-n}$ under the assumption that $m \geq 2n$ holds. In
 \cite{kitazawa 5}, as new answers, we construct new round fold maps on closed smooth manifolds having the
 structures of smooth bundles over (exotic) spheres
 or more general manifolds. More precisely, we considered round fold maps having good algebraic and differential topological
 properties and bundles over the source manifolds of the maps with good topological properties and constructed new round fold
 maps on the bundles by using operations compatiable with the structures of bundles ({\it P-operations}). For example, we constructed round fold maps on manifolds having the structures of $S^1$-bundles ({\it circle bundles})
 over some smooth manifolds and obtain ones into ${\mathbb{R}}^n$ ($n \geq 2$) on $S^1$-bundles
 over manifolds represented as connected
 sums of smooth bundles over $S^n$ with fibers diffemorphic
 to $S^{m-n}$ under the assumption that $m \geq 2n$ and $m-n>2$ hold. \\
\ \ \ In this paper, we extend results of \cite{kitazawa 5} and give some applications. More precisely, we construct
 new examples of round fold maps on smooth manifolds
 having the structures of $S^1$-bundles over some smooth manifolds under weaker conditions
 by using methods similar to ones used in \cite{kitazawa 5}. In this paper, we also study algebraic and differential topological properties of obtained maps and their source manifolds precisely. \\
\ \ \ This paper is organized as the following. \\ 
\ \ \ In section \ref{sec:2}, we recall {\it round} fold maps and some terminologies
 on round fold maps such as {\it axes} and {\it proper cores}. We also introduce classes of round fold maps such as {\it $C^{\infty}$ L-trivial} round
 fold maps and {\it $C^{\infty}$ S-trivial} ones. A {\it $C^{\infty}$ L-trivial} round fold map is defined as
 a round fold map satisfying
 a kind of triviality on the topological type around each connected component of the singular value set and a {\it $C^{\infty}$ S-trivial} round fold map is
 defined as one satisfying a kind of triviality on the topological type around the singular value set. Such classes were first
 introduced in \cite{kitazawa 2} and \cite{kitazawa 3} by the author and the original definitions are a bit different from those of the present paper. We
 introduce examples
 of round fold maps in Example \ref{ex:1}. \\
\ \ \ In section \ref{sec:3}, we recall {\it P-operations} defined in \cite{kitazawa 5}, which are operations used to construct new round fold maps
 on manifolds having the structures of bundles over smooth manifolds admitting $C^{\infty}$ L-trivial round
 fold maps. We apply these operations
 to construct new round fold maps on $S^1$-bundles over manifolds admitting $C^{\infty}$ L-trivial round
 fold maps and prove Theorems \ref{thm:1}, \ref{thm:2}, \ref{thm:3} and \ref{thm:4}, which are regarded as extensions of results in \cite{kitazawa 5}, which are introduced as Propositions \ref{prop:3} and \ref{prop:4} in the present paper. \\ 
\ \ \ In section \ref{sec:4}, as applications, we construct new round fold maps by applying Propositions and Theorems on P-operations obtained in section \ref{sec:3} to round maps introduced in Example \ref{ex:1}. Furthermore, we study the local and global structures of the
 obtained maps. For example, we study problems whether these maps being $L$-trivial are $S$-trivial or not. We also study the diffeomorphism types of the soruce manifolds of obtained maps. \\  
\ \ \ In this paper, for a smooth map $c$, the {\it fiber}
 of a point in the target manifold means the inverse image of the point by the map $c$. A {\it regular fiber} means a fiber of a
 regular value of a smooth map. We define the {\it singular set} of a smooth map $c$ as the set consisting of all the singular points of $c$ and denote the set by $S(c)$. For the map $c$, the
 set $c(S(c))$ is said to be the {\it singular value set} of $c$ and the set ${\mathbb{R}}^n-c(S(c))$ is said to be the {\it regular value set} of $c$. \\
\ \ \ Throughout this paper, all the manifolds and maps between them are assumed to be smooth and of class $C^{\infty}$ unless otherwise stated. Last, in
 this paper, We assume that $M$ is a closed manifold of dimension $m$, that $N$ is a manifold of dimension $n$ without boundary
 and that $m \geq n \geq 1$ holds.

\section{Round fold maps}
\label{sec:2}

In this section, we review round fold maps. See also \cite{kitazawa 2} and \cite{kitazawa 3} for example.

\begin{Def}[round fold maps (\cite{kitazawa 3})]
\label{def:1}
$f:M \rightarrow {\mathbb{R}}^n$ ($m \geq n \geq 2$) is said to be a {\it round} fold map if $f$ is $C^{\infty}$ equivalent to
 a fold map $f_0:M \rightarrow {\mathbb{R}}^n$ such that the followings hold.

\begin{enumerate}
\item The singular set $S(f_0)$ is a disjoint union of ($n-1$)-dimensional standard spheres and consists of $l \in \mathbb{N}$ connected components.
\item The restriction map $f_0 {\mid}_{S(f_0)}$ is an embedding.
\item Let ${D^n}_r:=\{(x_1,\cdots,x_n) \in {\mathbb{R}}^n \mid {\sum}_{k=1}^{n}{x_k}^2 \leq r \}$. Then $f_0(S(f_0))={\sqcup}_{k=1}^{l} \partial {D^n}_k$ holds.  
\end{enumerate}

We call $f_0$ a {\it normal form} of $f$. We call a ray $L$ from $0 \in {\mathbb{R}}^n$ an {\it axis} of $f_0$ and
 ${D^n}_{\frac{1}{2}}$ the {\it proper core} of $f_0$. Suppose that for a round fold map $f$, its normal form $f_0$ and diffeomorphisms
 $\Phi:M \rightarrow M$ and $\phi:{\mathbb{R}}^n \rightarrow {\mathbb{R}}^n$, the relation $\phi \circ f=f_0 \circ \Phi$ holds. Then
 for an axis $L$ of $f_0$, we also call ${\phi}^{-1}(L)$ an {\it axis} of $f$ and for the proper core ${D^n}_{\frac{1}{2}}$ of $f_0$, we
 also call ${\phi}^{-1}({D^n}_{\frac{1}{2}})$ a {\it proper core} of $f$. 
\end{Def}

For a round fold map $f:M \rightarrow {\mathbb{R}}^n$ and a connected component of the singular value set $C$, there exists
 a closed tubular neighborhood $N(C)$, which is regarded as a trivial $[-1,1]$-bundle over $C$ ($C$ corresponds to the image $C \times \{0\} \subset C \times [-1,1]$ of the section of the
 trivial bundle $N(C)$ regarded as $C \times [-1,1]$), such that the composition of the map $f {\mid}_{f^{-1}(N(C))}:f^{-1}(N(C)) \rightarrow N(C)$ and the projection to each connected
 component of the boundary $\partial N(C)$ gives $f^{-1}(N(C))$ the structure of a bundle. In the case where the inverse image of a connected
 component $C^{\prime}$ of the boundary $\partial N(C)$ is non-empty, the inverse image is a disjoint union of connected components
 of the boundary $\partial f^{-1}(N(C))$ and if we restrict the map giving $f^{-1}(N(C))$ the structure of
 a bundle over the connected component $C^{\prime}$ to the set $f^{-1}({C}^{\prime})$, then it coincides with the map $f$ on $f^{-1}(C^{\prime})$. \\
\ \ \ For example, if a connected component $C$ is the image of a connected component of the singular set consisting of singular points of index $0$, then the composition of the map $f {\mid}_{f^{-1}(N(C))}:f^{-1}(N(C)) \rightarrow N(C)$ and the projection to each connected
 component of the boundary $\partial N(C)$ gives $f^{-1}(N(C))$ the structure of a bundle whose fiber is a standard closed disc and whose structure group is a
 subgroup of linear transformations on the disc. See also \cite{saeki 2} for example. \\

\begin{Def}
\label{def:2}
Let $f:M \rightarrow {\mathbb{R}}^n$ be a round fold map. Assume that for any connected component $C$ of $f(S(f))$ and a small closed tubular neighborhood $N(C)$ of $C$ such that $\partial N(C)$ is the disjoint
 union of two connected components $C_a$ and $C_b$, $f^{-1}(N(C))$ has the structures
 of isomorphic trivial bundles whose fibers are diffeomorphic to a closed manifold $F_C$ over $C_a$ and $C_b$ and $f {\mid}_{f^{-1}(C_a)}:f^{-1}(C_a) \rightarrow C_a$ and $f {\mid}_{f^{-1}(C_b)}:f^{-1}(C_b) \rightarrow C_b$ give the structures of
 subbundles of the bundles $f^{-1}(N(C))$ {\rm (}$f^{-1}(C_a)$ or $f^{-1}(C_b)$ may be empty{\rm )}. Then, $f$ is said
 to be {\it $C^{\infty}$ L-trivial}. We call a fiber $F_C$ of the bundles $f^{-1}(N(C))$ a {\it normal fiber of $C$ corresponding
 to the bundles $f^{-1}(N(C))$}. \\
\end{Def}

\ \ \ Let $f$ be a normal form of a round fold map and $P_1:={D^{n}}_{\frac{1}{2}}$. We set $E:=f^{-1}(P_1)$ and $E^{\prime}:=M-f^{-1}({\rm Int} P_1)$. We set
 $F:=f^{-1}(p)$
 for $p \in \partial P_1$. We put $P_2:={\mathbb{R}}^n-{\rm Int} P_1$. Let $f_1:=f {\mid}_{E}:E \rightarrow P_1$ if $F$ is
 non-empty and let $f_2:=f {\mid}_{{E}^{\prime}}:{E}^{\prime} \rightarrow P_2$. \\
\ \ \ $f_1$ gives the structure of a trivial bundle over $P_1$ and ${f_1} {\mid}_{\partial E}:\partial E \rightarrow \partial P_1$ gives
 the structure of a trivial bundle
 over $\partial P_1$ if $F$ is non-empty. $f_2 {\mid}_{\partial {E}^{\prime}}:\partial {E}^{\prime} \rightarrow \partial P_2$
 gives the structure of a trivial bundle over $\partial P_2$ if $F$ is non-empty. \\
\ \ \ We can give ${E}^{\prime}$ the structures of bundles over $\partial P_2$ as follows. \\ 
\ \ \ Since for ${\pi}_P(x):=\frac{1}{2} \frac{x}{|x|}$ ($x \in P_2$), ${\pi}_P \circ f {\mid}_{{E}^{\prime}}$ is
 a proper submersion, this map gives ${E}^{\prime}$ the structure of
 a $C^{\infty}$ $f^{-1}(L)$-bundle over $\partial P_2$ (apply Ehresmann's fibration
 theorem \cite{ehresmann}). \\
\ \ \ Considering these observations, we introduce the following classes. For the classes, see also \cite{kitazawa 2}, in which such classes were
 introduced first. These classes are a bit different from ones here. 

\begin{Def}
\label{def:3}
Let $f:M \rightarrow {\mathbb{R}}^n$ be a round fold map. \\
\ \ \ For a proper core $P$ and an axis $L$ of $f$, We give the manifold $f^{-1}({\mathbb{R}}^n-{\rm Int} P)$ the structure of
 a bundle over $\partial P$ whose fiber is diffeomorphic to $f^{-1}(L)$ such
 that $f {\mid}_{f^{-1}(\partial P)}:f^{-1}(\partial P) \rightarrow \partial P$ gives the structure of a subbundle
 of the previous bundle $f^{-1}({\mathbb{R}}^n-{\rm Int} P)$ ($f^{-1}(\partial P)$ may be empty). We call
 such a bundle a {\it surrounding bundle} of $f$. \\
\ \ \ If a surrounding bundle is a trivial bundle, then
 $f$ is said to be {\it $C^{\infty}$ S-trivial}.
%
%
\end{Def}

We easily obtain a round fold map which is $C^{\infty}$ L-trivial and $C^{\infty}$ S-trivial as shown in the following (see \cite{kitazawa 2} and \cite{kitazawa 3} for example). \\
\ \ \ Let $\bar{M}$ be a compact manifold with non-empty boundary $\partial \bar{M}$. Let $a \in \mathbb{R}$. Then, there exists a Morse
 function $\tilde{f}:\bar{M} \rightarrow [a,+\infty)$ satisfying the followings. 
\begin{enumerate}
\item $a$ is the minimum of $\tilde{f}$ and ${\tilde{f}}^{-1}(a)=\partial \bar{M}$ holds.
\item All the singular points of $\tilde{f}:\bar{M} \rightarrow [a,+\infty)$ are in $\bar{M}-\partial \bar{M}$ and at distinct singular points, the values are always distinct.
\end{enumerate}
Let $\Phi:\partial (\bar{M} \times \partial ({\mathbb{R}}^n-{\rm Int} D^n)) \rightarrow \partial (\partial \bar{M} \times D^n)$
 and $\phi:\partial ({\mathbb{R}}^n-{\rm Int} D^n) \rightarrow \partial D^n$ be diffeomorphisms. Let $p_1:\partial \bar{M} \times \partial ({\mathbb{R}}^n-{\rm Int} D^n) \rightarrow \partial ({\mathbb{R}}^n-{\rm Int} D^n)$ and
 $p_2:\partial \bar{M} \times \partial D^n \rightarrow \partial D^n$ be
 the canonical projections. Suppose that the following diagram commutes. 

$$
\begin{CD}
\partial \bar{M} \times \partial ({\mathbb{R}}^n-{\rm Int} D^n)  @> \Phi >> \partial \bar{M} \times \partial D^n \\
@VV p_1 V @VV p_2 V \\
\partial ({\mathbb{R}}^n-{\rm Int} D^n) @> \phi >> \partial D^n
\end{CD}
$$

By using the diffeomorphism $\Phi$, we construct $M:=(\partial \bar{M} \times D^n) {\bigcup}_{\Phi} (\bar{M} \times \partial ({\mathbb{R}}^n-{\rm Int} D^n))$. 
Let $p:\partial \bar{M} \times D^n \rightarrow D^n$ be the canonical projection. Then gluing the two
 maps $p$ and $\tilde{f} \times {{\rm id}}_{S^{n-1}}$ together by using the two diffeomorphisms $\Phi$ and $\phi$,
 we obtain a round fold map $f:M \rightarrow {\mathbb{R}}^n$. \\
\ \ \ If $\bar{M}$ is a compact manifold without boundary, then there exists a Morse
 function $\tilde{f}:\bar{M} \rightarrow [a,+\infty)$ such that $\tilde{f}(\bar{M}) \subset (a,+\infty)$ and that at distinct singular points, the values are always distinct. We are enough
 to consider $\tilde{f} \times {\rm id}_{S^{n-1}}$
 and embed $[a,+\infty) \times S^{n-1}$ into ${\mathbb{R}}^n$ to construct a round fold map whose source manifold is $\bar{M} \times S^{n-1}$. \\
\ \ \ We call this construction of a round fold map a {\it trivial spinning construction}. 

We introduce some known examples of round fold maps with their source manifolds.

\begin{Ex}
\label{ex:1}
\begin{enumerate}
\item
\label{ex:1.1}
Special generic maps with singular sets diffeomorphic to standard spheres into Euclidean spaces whose dimensions are larger than $1$ such that the restriction maps
 to the singular sets are embeddings are round fold maps. Such maps are $C^{\infty}$ L-trivial and $C^{\infty}$ S-trivial. The source manifolds of such maps are
 always homotopy spheres. Let $M$ be a closed $m$-dimensional manifold admitting such a map $f:M \rightarrow {\mathbb{R}}^n$ ($m \geq n \geq 2$). For a small closed tubular neighborhood $N(f(S(f)))$ of the set $f(S(f))$, $f^{-1}(N(f(S(f))))$ has the structures of trivial bundles as in Definiiton \ref{def:2} such that
 a normal fiber $F_{f(S(f))}$ of $f(S(f))$ corresponding to the bundles $f^{-1}(N(f(S(f))))$ is diffeomorphic to the standard closed disc $D^{m-n+1}$. \\
\ \ \ The $m$-dimensional standard sphere $S^m$ admits such a map into ${\mathbb{R}}^n$ for $m \geq n \geq 2$. In section 5 of \cite{saeki 2}, on every homotopy sphere whose dimension is
 $m$ ($2 \leq m <4$, $m>4$) and the $4$-dimensional standard sphere $S^4$, such maps into ${\mathbb{R}}^2$ are constructed. 
\item 
\label{ex:1.2}
[\cite{kitazawa 2}, \cite{kitazawa 4}, \cite{kitazawa 5}]
Let $F \neq \emptyset$ be a closed and connected manifold. Let $M$ have the structure of an $F$-bundle
 over $S^n$. Then, $M$ admits a round fold map $f:M \rightarrow {\mathbb{R}}^n$ such that the followings hold.
\begin{enumerate}
\item $f$ is $C^{\infty}$ S-trivial.
\item For an axis $L$ of $f$, $f^{-1}(L)$ is diffeomorphic to $F \times [0,1]$.
\item Two connected components of the fiber of a point in a proper core of $f$ is regarded as fibers of the $F$-bundle
 over $S^n$.
\end{enumerate}
\item
\label{ex:1.3}[\cite{kitazawa}, \cite{kitazawa 2}, \cite{kitazawa 4}]
Let $m, n \in \mathbb{N}$, $n \geq 2$ and $m \geq 2n$. \\
\ \ \ Let $M$ be an $m$-dimensional manifold represented as the connected sum of $l \in \mathbb{N}$ closed oriented manifolds having
 the sturctures of $S^{m-n}$-bundles over $S^n$ admits
 a $C^{\infty}$ L-trivial round fold map $f$ into ${\mathbb{R}}^n$ such that the followings hold.
\begin{enumerate}
\item All the regular fibers of $f$ are disjoint unions of finite copies of $S^{m-n}$.
\item The number of connected components of the singular set $S(f)$ and the number of connected components of the fiber of a point in a proper core of $f$ are $l$.
\item For any connected component $C$ of $f(S(f))$ and a small closed tubular neighborhood $N(C)$ of $C$, $f^{-1}(N(C))$ has the structures of trivial
 bundles as in Definition \ref{def:2} such that a normal
 fiber $F_C$ of $C$ corresponding to the bundles $f^{-1}(N(C))$ is diffeomorphic to a disjoint union of a finite number of the following manifolds.
\begin{enumerate}
\item The {\rm (}$m-n+1${\rm )}-dimensional standard closed disc $D^{m-n+1}$.
\item The standard {\rm (}$m-n+1${\rm )}-dimensional sphere $S^{m-n+1}$ with the interior of a union of disjoint three {\rm (}$m-n+1${\rm )}-dimensional standard
 closed discs removed.
\end{enumerate}
\item All the connected components of the fiber of a point in a proper core of $f$ are regarded as fibers of the $S^{m-n}$-bundles
 over $S^n$ and a fiber of any $S^{m-n}$-bundle over $S^n$ appeared in the connected sum is regarded as a connected component of the
 fiber of a point in a proper core of $f$.  
\end{enumerate}
\item
\label{ex:1.4}
Special generic maps into ${\mathbb{R}}^n$ whose singular sets are disjoint unions of two standard spheres such that the restriction maps
 to the singular sets are embeddings are round fold maps. Let $M$ be a closed $m$-dimensional
 manifold admitting such a map $f:M \rightarrow {\mathbb{R}}^n$ ($m \geq n \geq 2$). Since $f$ is special generic, for each connected component $C$ of $f(S(f))$, there exists a small closed tubular neighborhood $N(C)$ such that $f^{-1}(N(C))$ has the structures
 of isomorphic trivial bundles whose fibers are diffeomorphic to the standard closed disc $D^{m-n+1}$ as in Definition \ref{def:2}. \\
\ \ \ We easily know that the $m$-dimensional manifolds admitting such maps into ${\mathbb{R}}^n$ ($m \geq n \geq 2$) being $C^{\infty}$ L-trivial are always
 homeomorphic to $S^{n-1} \times S^{m-n+1}$ and that $S^{n-1} \times S^{m-n+1}$ admits such a map being $C^{\infty}$ S-trivial and $C^{\infty}$ L-trivial
 by virtue of a trivial spinning construction. Moreover, for example, in the
 case where $m-n=0,1,2,3$ holds, the source manifold of such
 a map is diffeomorphic to the product $S^{n-1} \times S^{m-n+1}$. It easily follows from the fact that the diffeomorphism groups of the
 spheres $S^1$, $S^2$ and $S^3$ are homotopy equivalent to the groups of all the linear transformations on the spheres (\cite{hatcher}, \cite{smale}). \\
\ \ \ For explicit theories of such special generic maps, see also \cite{saeki 2}.
\end{enumerate}
\end{Ex}


\section{P-operations and constructions of round fold maps of new types on manifolds having the structures of $S^1$-bundles}
\label{sec:3}
\subsection{P-operations}
The following proposition has been shown in \cite{kitazawa 5}.

\begin{Prop}[\cite{kitazawa 5}]
\label{prop:1}
Let $M$ be a closed manifold of dimension $m$ and $f:M \rightarrow {\mathbb{R}}^n$ be a $C^{\infty}$ L-trivial round
 fold map {\rm (}$m \geq n \geq 2${\rm )}. Let $F$ {\rm (}$\neq \emptyset${\rm )} be a closed manifold and ${M}^{\prime}$ be a closed manifold having the structure of an $F$-manifold
 over $M$ such that for any connected component $C$ of $f(S(f))$ and a small closed tubular neighborhood $N(C)$ of $C$, the
 restriction to $f^{-1}(N(C))$ is a trivial bundle. Then on ${M}^{\prime}$ there exists a $C^{\infty}$ L-trivial round
 fold map $f^{\prime}:{M}^{\prime} \rightarrow {\mathbb{R}}^n$. 
\end{Prop}
In the proof of this proposition, we use the notation ${D^n}_r:=\{(x_1,\cdots,x_n) \in {\mathbb{R}}^n \mid {\sum}_{k=1}^{n}{x_k}^2 \leq r \}$ appeared in Definition \ref{def:1}. 
\begin{proof}[Proof of Proposition \ref{prop:3}]
\ \ \ We assume that $f(M)$ is diffeomorphic to $D^n$. We can prove the theorem similarly if $f(M)$ is not diffeomorphic to $D^n$. \\
\ \ \ We may assume that $f:M \rightarrow {\mathbb{R}}^n$ is a normal form. Let $S(f)$ consist of $l$ connected
 components. Set $P_0:={D^n}_{\frac{1}{2}}$ and $P_k:={D^n}_{k+\frac{1}{2}}-{\rm Int} {D^n}_{k-\frac{1}{2}}$ for an integer $1 \leq k \leq l$. Then ${f}^{-1}(P_k)$ has the
 structure of a trivial $C^{\infty}$ bundle
 over $\partial {D^n}_{k-\frac{1}{2}}$ or $\partial {D^n}_{k+\frac{1}{2}}$ with fibers diffeomorphic to a compact $C^{\infty}$ manifold, which we denote by $E_k$, such
 that $f {\mid}_{{f^{-1}}(\partial {D^n}_{k-\frac{1}{2}})}$ and $f {\mid}_{f^{-1}(\partial {D^n}_{k+\frac{1}{2}})}$ give the
 structures of subbundles (we denote the fibers of two subbundles by ${E_k}^{1} \subset E_k$ and ${E_k}^2 \subset E_k$, respectively) for $1 \leq k \leq l$. For any integer $1 \leq k \leq l$ and a diffeomorphism ${\phi}_k$ from
 $f^{-1}(\partial {D^n}_{k-\frac{1}{2}}) \subset f^{-1}(P_k)$ onto $f^{-1}(\partial {D^n}_{k-\frac{1}{2}}) \subset f^{-1}(P_{k-1})$ regarded as a
 bundle isomorphism between the two trivial bundles over standard spheres inducing the identification between the base spaces, $M$ is regarded as $(\cdots ((f^{-1}({D^n}_{\frac{1}{2}})) {\bigcup}_{{\phi}_1} f^{-1}(P_1)) \cdots) {\bigcup}_{{\phi}_l} f^{-1}(P_l)$
 and for any integer $1 \leq k \leq l$ and a diffeomorphism ${\Phi}_k$ from
 $f^{-1}(\partial {D^n}_{k-\frac{1}{2}}) \times F \subset f^{-1}(P_k) \times F$ onto $f^{-1}(\partial {D^n}_{k-\frac{1}{2}}) \times F \subset f^{-1}(P_{k-1}) \times F$ regarded as a
 bundle isomorphism between the two trivial $F$-bundles inducing ${\phi}_k$, ${M}^{\prime}$ is
 regarded as
 $(\cdots ((f^{-1}({D^{n}}_{\frac{1}{2}}) \times F) {\bigcup}_{{\Phi}_1} (f^{-1}(P_1) \times F)) \cdots) {\bigcup}_{{\Phi}_l} ({f^{-1}}(P_l) \times F)$. We construct a map on $f^{-1}(P_k) \times F$. This manifold has the structure of a
 trivial bundle over $\partial {D^n}_{k-\frac{1}{2}}$ or $\partial {D^n}_{k+\frac{1}{2}}$ with fibers diffeomorphic to $E_k \times F$. On $E_k \times F$ there exists
 a Morse function $\tilde{f_k}$ such that the followings hold.

\begin{enumerate}
\item $\tilde{f_k}(E_k \times F) \subset [k-\frac{1}{2},k+\frac{1}{2}]$ and $\tilde{f_k}({\rm Int} (E_k \times F)) \subset (k-\frac{1}{2},k+\frac{1}{2})$ hold. 
\item $\tilde{f_k}({E_k}^1 \times F)=\{k-\frac{1}{2}\}$ holds if ${E_k}^1 \times F$ is non-empty. 
\item $\tilde{f_k}({E_k}^2 \times F)=\{k+\frac{1}{2}\}$ holds if ${E_k}^2 \times F$ is non-empty.
\item Singular points of $\tilde{f_k}$ are in the interior of $E_k \times F$ and at two distinct singular points, the values are always distinct.
\end{enumerate}

We obtain a map ${\rm id}_{S^{n-1}} \times \tilde{f_k}:S^{n-1} \times E_k \times F \rightarrow S^{n-1} \times [k-\frac{1}{2},k+\frac{1}{2}]$. We can identify $S^{n-1} \times [k-\frac{1}{2},k+\frac{1}{2}]$
 with $P_k={D^n}_{k+\frac{1}{2}}-{\rm Int} {D^n}_{k-\frac{1}{2}}$ by identifying $(p,t) \in S^{n-1} \times [k-\frac{1}{2},k+\frac{1}{2}]$ with $tp \in P_k$ where we regard $S^{n-1}$ as the unit sphere of dimension $n-1$. 
 By gluing the composition of the projection from $f^{-1}({D^n}_{\frac{1}{2}}) \times F$ onto $f^{-1}({D^n}_{\frac{1}{2}})$ and $f {\mid}_{f^{-1}({D^n}_{\frac{1}{2}})}:f^{-1}({D^n}_{\frac{1}{2}}) \rightarrow {D^n}_{\frac{1}{2}}$ and the family
 $\{{\rm id}_{S^{n-1}} \times \tilde{f_k}\}$ together by using the family $\{{\Phi}_k \}$ and the family of identifications in the target manifold ${\mathbb{R}}^n$, we obtain
 a new round fold map ${f}^{\prime}:M^{\prime} \rightarrow {\mathbb{R}}^n$. 
\end{proof}

In the proof, from $f:M \rightarrow {\mathbb{R}}^n$, we obtain ${f}^{\prime}:M^{\prime} \rightarrow {\mathbb{R}}^n$. We call the operation of
 constructing ${f}^{\prime}$ from $f$ a {\it P-operation by $F$ to $f$}. \\
\ \ \ For example, for any closed manifold $F \neq \emptyset$ and on any manifold admitting the structure of a $F$-bundle over the $n$-dimensional ($n \geq 2$) standard sphere, we obtain a round fold map as
 introduced in Example \ref{ex:1} (\ref{ex:1.2}) by a P-operation by $F$ to a special generic map from $S^n$ into ${\mathbb{R}}^n$ presented in Example \ref{ex:1} (\ref{ex:1.1}). 

\subsection{P-operations by $S^1$ to round fold maps and their applications}
 We apply $P$-operations by $S^1$ to obtain round fold maps which do not appear in \cite{kitazawa 5}. Before this, we
 recall fundamental terms and facts on $S^1$-bundles. In this paper, for a topological space $X$, we denote the $k$-th (co)homology
 group of $X$ whose coefficient ring is $R$ by $H_k(X;R)$ (resp. $H^k(X;R)$). \\
\ \ \ It is well-known that $S^1$-bundles are regarded as bundles whose structure groups consist of linear transformations on the circle and regarded as the 2nd orthogonal group $O(2)$. An $S^1$-bundle is said
 to be {\it orientable} if the structure group consist of orientation preserving linear transformations and regarded as the 2nd rotation group $SO(2) \subset O(2)$. We can consider
 {\it Stiefel-Whitney classes} of $S^1$-bundles, which are cohomology classes of the base spaces whose coefficient rings are $\mathbb{Z}/2\mathbb{Z}$, including their {\it 1st Stiefel-Whitney classes} ({\it 2nd Stiefel-Whitney classes}), which are 1st cohomology classes (resp. 2nd cohomology classes) of the base spaces with coefficients $\mathbb{Z}/2\mathbb{Z}$. If an $S^1$-bundle is
 orientable, then we can orient the bundle and consider its {\it Euler class}, which is a 2nd cohomology class of the base space whose coefficient ring is $\mathbb{Z}$. \\
\ \ \ For an oriented $S^1$-bundle over a $2$-dimensional oriented closed and connected manifold, we can consider the pairing of the Euler class and the fundamental class of the
 base space and call the obtained integer the {\it Euler number} of the bundle. 
\ \ \ We introduce
 known facts on classifications of $S^1$-bundles without proofs.

\begin{Prop}
\label{prop:2}
Let $X$ be a topological space.
\begin{enumerate}
\item
\label{prop:2.1}
 The 1st Stiefel-Whitney class $\alpha \in H^1(X;\mathbb{Z}/2\mathbb{Z})$ of an $S^1$-bundle over $X$ vanishes if and only if the bundle is orientable.
\item
\label{prop:2.2}
 For any $\alpha \in H^2(X;\mathbb{Z})$, there exists an oriented $S^1$-bundle over $X$ whose Euler class is $\alpha$. 
\item
\label{prop:2.3}
 Two oriented $S^1$-bundles on $X$ are isomorphic if the Euler classes are same.
\item
\label{prop:2.4}
 Two oriented $S^1$-bundles over a $2$-dimensional oriented closed and connected manifold are isomorphic if and
 only if the Euler numbers of the bundles coincide. Furthermore, manifolds admitting the structures
 of $S^1$-bundles over a $2$-dimensional oriented closed and connected manifolds are homeomorphic {\rm (}diffeomorphic{\rm )} if and
 only if the absolute values of the Euler numbers of the bundles coincide. The 1st homology group and the 2nd cohomology group of the
 manifold admitting an oriented $S^1$-bundle such that the absolute value of the Euler number is $k \in \mathbb{N} \bigcup \{0\}$ with coefficient
 ring $\mathbb{Z}$ are isomorphic to the direct sum of the 1st
 {\rm (}co{\rm )}homology group of the base space with coefficient ring $\mathbb{Z}$ and $\mathbb{Z}/k\mathbb{Z}$. The latter part
 of this direct sum is generated by the 1st homology class represented by a fiber of the bundle. Furthermore, for any section of
 the restriction of the bundle to any 1-skelton representing some element of 1st homology group of the base space with coefficient
 ring $\mathbb{Z}$, the image of the section represents an element of the 1st homology group before and if we consider
 the direct sum here, the 1st homology class is of the form $(x.0)$.   
%
\end{enumerate}
\end{Prop}
 For the theory of bundles whose structure groups consist of linear transformations including $S^1$-bundles and vector bundles and their
 characteristic classes including Stiefel-Whitney classes and Euler classes, see also \cite{milnor stasheff} for example. \\

Let $M$ be a closed manifold of
 dimension $m$ and let $f:M \rightarrow {\mathbb{R}}^n$ {\rm (}$m \geq n \geq 2${\rm )} be a $C^{\infty}$ L-trivial
 round fold map. For any connected component $C$ of $f(S(f))$, we denote by $N(C)$ a
 small closed tubular neighborhood of $C$ such that the
 inverse image has the structures of trivial bundles
 as in Definition \ref{def:2} and by $F_C$ a normal fiber of $C$ corresponding to the bundles $f^{-1}(N(C))$. \\
\ \ \ We now construct a graph corresponding to $f_F:=(f,\{F_C\}_{C \in {\pi}_0 (S(f))})$, where $C \in {\pi}_0 (f(S(f)))$ represents
 the element of the set ${\pi}_0 (f(S(f)))$ of all the connected components of $f(S(f))$. \\
\ \ \ Let $V_{f_F}$ be the set of all the closures of connected components of the regular value set ${\mathbb{R}}^n-f(S(f))$.
 Let $R$ be a commutative ring. We define the set $E^{(H_k,R)}_{f_F} \subset V_{f_F} \times V_{f_F}$ so that the followings hold.
\begin{enumerate}
\item The set $E^{(H_k,R)}_{f_F}$ consists of elements $(P_1,P_2)$ of the set $V_{f_F} \times V_{f_F}$ such that the intersection $P_1 \bigcap P_2$ is not empty. 
\item A pair $(P_1,P_2)$ of closures of connected components of the
 space ${\mathbb{R}}^n-f(S(f))$ such that $P_1 \bigcap P_2$ is not empty and that $P_1$ is in the bounded connected component of ${\mathbb{R}}^n-{\rm Int} P_2$, $(P_1,P_2) \in E_{f_F}$ holds if and only if for
 the connected component $C=P_1 \bigcap P_2$ of $f(S(f))$, the homomorphism
 from $H_k({\partial}_i F_C;R)$ into $H_k(F_C;R)$
 induced by the natural inclusion is surjective.
\item A pair $(P_1,P_2)$ of closures of connected components of the
 space ${\mathbb{R}}^n-f(S(f))$ such that $P_1 \bigcap P_2$ is not empty
 and that $P_2$ is in the bounded connected component of ${\mathbb{R}}^n-{\rm Int} P_1$, $(P_1,P_2) \in E_{f_F}$ holds if and only if for
 the connected component $C:=P_1 \bigcap P_2$ of $f(S(f))$, the homomorphism
 from $H_k({\partial}_o F_C;R)$ into $H_k(F_C;R)$
 induced by the natural inclusion is surjective.
\end{enumerate}

Thus, we obtained a graph $(V_{f_F},E^{(H_k,R)}_{f_F})$ and we call it the {\it $(H_k,R)$-graph} of $f_F=(f,\{F_C\}_{C \in {\pi}_0 (f(S(f)))})$. 
\begin{Def}
\label{def:4}
For the $(H_k,R)$-graph of $f_F=(f,\{F_C\}_{C \in {\pi}_0 (f(S(f)))})$, a non-empty subset $S \subset V_{f_F}$ is said to
 be a {\it starting set} of the graph if for any $x \in V_{f_F}$, there exists a path from a vertex in $S$ to $x$. \\
\ \ \ Let $R_i$ be a commutative ring and let $k_i$ be an integer for $i=1,2$. For the $(H_{k_1},R_1)$-graph and
 the $(H_{k_2},R_2)$-graph of $f_F=(f,\{F_C\}_{C \in {\pi}_0 (f(S(f)))})$, the $(H_{k_1},R_1)$-graph and the $(H_{k_2},R_2)$-graph are start-equivalent on the set $S$ if for any vertex $s \in V_{f_F}$, there exists a vertex $s^{\prime} \in S$ such
 that there exists a path of both of these graphs from $s^{\prime}$ to $s$.
\end{Def}

\ \ \ We show the following theorem. 

\begin{Thm}
\label{thm:1}
Let $M$ be a closed manifold of
 dimension $m$ and $f:M \rightarrow {\mathbb{R}}^n$ {\rm (}$m \geq n \geq 2${\rm )} be a $C^{\infty}$ L-trivial round fold map. For any connected component $C$ of $f(S(f))$, we denote by $N(C)$  a small closed tubular neighborhood of $C$ such that the
 inverse image has the structures of trivial bundles as in Definition \ref{def:2} and by $F_C$ a normal
 fiber of $C$ corresponding to the bundles $f^{-1}(N(C))$ over $C$. Let $F_C$ satisfy
 $H^{2}(F_C;\mathbb{Z}) \cong \{0\}$. Let $S$ be a starting set of the
$(H_1,\mathbb{Z}/2\mathbb{Z})$-graph of $f_F=(f,\{F_C\}_{C \in {\pi}_0 (f(S(f)))})$. Let $F_s$ be the regular fiber of a point in a connected component of ${\mathbb{R}}^n-f(S(f))$ whose
 closure is represented as a vertex $s \in S$ of the graph above and let $i_s:F_s \rightarrow M$ be the natural inclusion. Then, we have the followings.
 
\begin{enumerate}
\item Let $n \geq 4$ hold. Then, by a P-operation to $f$, for any
 closed manifold $M^{\prime}$ having the structure of an $S^1$-bundle over $M$ whose
 1st Stiefel-Whitney class vanishes on ${i_s}_{\ast}(H_1(F_s;\mathbb{Z}/2\mathbb{Z})) \subset H_1(M;\mathbb{Z}/2\mathbb{Z})$ for any $s \in S$, we
 can obtain a round fold map $f^{\prime}:M^{\prime} \rightarrow {\mathbb{R}}^n$. 
\item Let $n=3$ hold. Furthermore, we assume that for any commutative ring $R$, the $(H_1,\mathbb{Z}/2\mathbb{Z})$-graph of $f_F=(f,\{F_C\}_{C \in {\pi}_0 (f(S(f)))})$
 and the $(H_0,R)$-graph of $f_F=(f,\{F_C\}_{C \in {\pi}_0 (f(S(f)))})$ are start-equivalent on the set $S$. 
Then, by a P-operation to $f$, for any closed manifold $M^{\prime}$ having the structure of an $S^1$-bundle over $M$ such that
 the followings hold, we
 can obtain a round fold map $f^{\prime}:M^{\prime} \rightarrow {\mathbb{R}}^n$. 
\begin{enumerate} 
\item The 1st
 Stiefel-Whitney class vanishes on ${i_s}_{\ast}(H_1(F_s;\mathbb{Z}/2\mathbb{Z})) \subset H_1(M;\mathbb{Z}/2\mathbb{Z})$ for any $s \in S$.
\item For $s \in S$, let ${C_s}^{\prime}$ be a 2-dimensional sphere in a closure of a connected component of ${\mathbb{R}}^3-f(S(f))$ represented by $s \in S$ such that for a connected component $C_s$ of $f(S(f))$ and the small closed
 tubular neighborhood $N(C_s)$ defined before, ${C_s}^{\prime}$ is a connected component of $\partial N(C_s)$. The restriction of the bundle
 to the image of a section of the trivial bundle given by $f {\mid}_{f^{-1}({C_s}^{\prime})}:f^{-1}({C_s}^{\prime}) \rightarrow {C_s}^{\prime}$ is trivial. 
\end{enumerate} 
\item  Let $n=2$ hold. Furthermore, we assume that for any commutative ring $R$, the $(H_1,\mathbb{Z}/2\mathbb{Z})$-graph of $f_F=(f,\{F_C\}_{C \in {\pi}_0 (f(S(f)))})$
 and the $(H_0,R)$-graph of $f_F=(f,\{F_C\}_{C \in {\pi}_0 (f(S(f)))})$ are start-equivalent on the set $S$ and that the $(H_1,\mathbb{Z}/2\mathbb{Z})$-graph of $f_F=(f,\{F_C\}_{C \in {\pi}_0 (f(S(f)))})$ and the $(H_1,\mathbb{Z})$-graph of $f_F=(f,\{F_C\}_{C \in {\pi}_0 (f(S(f)))})$
 are start-equivalent on the set $S$. \\
\ \ \ For $s \in S$, let ${C_s}^{\prime}$ be a circle in a closure of a connected component of ${\mathbb{R}}^2-f(S(f))$ represented by $s \in S$ such that for a connected component $C_s$ of $f(S(f))$ and the small closed
 tubular neighborhood $N(C_s)$ defined before, ${C_s}^{\prime}$ is a connected component of $\partial N(C_s)$. The restriction of the bundle
 to the total space of the trivial bundle given by $f {\mid}_{f^{-1}({C_s}^{\prime})}:f^{-1}({C_s}^{\prime}) \rightarrow {C_s}^{\prime}$ is orientable and trivial. \\
\ \ \ Then, by
 a P-operation to $f$, for any closed manifold $M^{\prime}$ having the structure of an $S^1$-bundle over $M$, we
 can obtain a round fold map $f^{\prime}:M^{\prime} \rightarrow {\mathbb{R}}^n$.  
\end{enumerate}
 
\end{Thm}
\begin{proof}
In all the cases, it suffices to prove that for any connected component $C \subset f(S(f))$, given $S^1$-bundles over $M$ are trivial
 on $f^{-1}(N(C))$. Let $C_0 \subset f(S(f))$ be a connected component which is in the closures of two connected component of ${\mathbb{R}}^n-f(S(f))$ represented by
 vertices $s_1,s_2 \in V_{f_F}$ of the graph and let $(s_1,s_2) \in E^{(H_1,\mathbb{Z}/2\mathbb{Z})}_{f_F}$. Let the boundary $\partial N(C_0)$ of $N(C_0)$ defined in the assumption of this theorem consist of two connected components $C_1$ and $C_2$ as in Definition \ref{def:2} and
 let $C_j$ be in the closure of a connected component of ${\mathbb{R}}^n-f(S(f))$ represented by $s_j$. In addition, we denote the fiber of the bundle $f^{-1}(C_j)$ over $C_j$ as in Definition \ref{def:2} by ${\partial}_j F_{C_0}$. \\ 
\ \ \ We prove the statement in the first case. The
 homomorphism from $H_1(f^{-1}(C_1);\mathbb{Z}/2\mathbb{Z}) \cong H_1({\partial}_1 F_{C_0};\mathbb{Z}/2\mathbb{Z})$ into
 $H_1(f^{-1}(N(C_0));\mathbb{Z}/2\mathbb{Z}) \cong H_1(F_{C_0};\mathbb{Z}/2\mathbb{Z})$ induced by
 the natural inclusion is surjective, since the homomorphism
 from $H_1({\partial}_1 F_{C_0};\mathbb{Z}/2\mathbb{Z})$ into $H_1(F_{C_0};\mathbb{Z}/2\mathbb{Z})$ induced by the natural inclusion is assumed to be surjective. The kernel of the homomorphism from $H^1(f^{-1}(N(C_0));\mathbb{Z}/2\mathbb{Z})$ into $H^1(f^{-1}(C_1);\mathbb{Z}/2\mathbb{Z})$ induced by the
 natural inclusion is $\{0\}$. $H^2(f^{-1}(N(C_0));\mathbb{Z}) \cong H^2(C_0 \times F_{C_0};\mathbb{Z}) \cong \{0\}$ holds since $H^2(F_{C_0};\mathbb{Z}) \cong \{0\}$ is assumed. Thus, if the
 restriction of an $S^1$-bundle over $M$ to $f^{-1}(C_1)$ is orientable, then the restriction of the original bundle to $f^{-1}(N({C_0}))$ is orientable and it is a trivial bundle. \\
\ \ \ Consider an ($n-1$)-dimensional sphere ${C_s}^{\prime}$ in a closure of a connected component of ${\mathbb{R}}^n-f(S(f))$ represented by $s \in S$ such that for a connected component $C_s$ of $f(S(f))$ and a small closed
 tubular neighborhood $N(C_s)$ as in Definition \ref{def:2}, ${C_s}^{\prime}$ is a connected component of $\partial N(C_s)$.
 $f^{-1}({C_s}^{\prime})$ is diffeomorphic to ${C_s}^{\prime} \times F_s$ and $C_s \times F_s$. $H^1(f^{-1}({C_s}^{\prime});\mathbb{Z}/2\mathbb{Z}) \cong H^1(C_s \times F_s;\mathbb{Z}/2\mathbb{Z}) \cong H^1(F_s;\mathbb{Z}/2\mathbb{Z})$ holds
 and if the 1st Stiefel-Whitney class of an $S^1$-bundle over $M$ vanishes on $F_s$, then the restriction
 of the $S^1$-bundle over $M$ to $f^{-1}({C_s}^{\prime})$ is trivial. Thus, by the discussion in the previous paragraph together with an induction based on the fact that $S$ is
 a starting set of the graph, for any
 connected component $C$ of $f(S(f))$, the restriction
 of any $S^1$-bundle over $M$ whose
 1st Stiefel-Whitney class vanishes on $i_{\ast}(H_1(F_s;\mathbb{Z}/2\mathbb{Z})) \subset H_1(M;\mathbb{Z}/2\mathbb{Z})$ ($s \in S$) to $f^{-1}(N(C))$ is
 orientable and it is a trivial bundle. This completes the proof of the first case. \\
\ \ \ We prove the statement in the second case. As in the first case, the
 homomorphism from $H_1(f^{-1}(C_1);\mathbb{Z}/2\mathbb{Z}) \cong H_1({\partial}_1 F_{C_0};\mathbb{Z}/2\mathbb{Z})$ into
 $H_1(f^{-1}(N(C_0));\mathbb{Z}/2\mathbb{Z}) \cong H_1(F_{C_0};\mathbb{Z}/2\mathbb{Z})$ induced by
 the natural inclusion is surjective, since the homomorphism
 from $H_1({\partial}_1 F_{C_0};\mathbb{Z}/2\mathbb{Z})$ into $H_1(F_{C_0};\mathbb{Z}/2\mathbb{Z})$ induced by the natural inclusion is assumed to be surjective. The kernel of the homomorphism from $H^1(f^{-1}(N(C_0));\mathbb{Z}/2\mathbb{Z})$ into $H^1(f^{-1}(C_1);\mathbb{Z}/2\mathbb{Z})$ induced by the
 natural inclusion is $\{0\}$. We have $H^2(f^{-1}(N({C_0}));\mathbb{Z}) \cong H^2(C_0 \times F_{C_0};\mathbb{Z}) \cong H^2(C_0;\mathbb{Z}) \otimes H^0(F_{C_0};\mathbb{Z}) \cong \mathbb{Z} \otimes H^0(F_{C_0};\mathbb{Z})$ since $H^2(F_{C_0};\mathbb{Z}) \cong \{0\}$. We
 also have $H^2(f^{-1}(C_1);\mathbb{Z}) \cong H^2(C_0 \times {\partial}_1 F_{C_0};\mathbb{Z}) \cong (H^2(C_0;\mathbb{Z}) \otimes H^0({\partial}_1 F_{C_0};\mathbb{Z})) \oplus (H^0(C_0;\mathbb{Z}) \otimes H^2({\partial}_1 F_{C_0};\mathbb{Z})) \cong (\mathbb{Z} \otimes H^0({\partial}_1 F_{C_0};\mathbb{Z})) \oplus (H^0(C_0;\mathbb{Z}) \otimes H^2({\partial}_1 F_{C_0};\mathbb{Z}))$ since $H^2(F_{C_0};\mathbb{Z}) \cong \{0\}$. Moreover, we have $H_2(f^{-1}(C_1);\mathbb{Z}) \cong H_2({\partial}_1 F_{C_0};\mathbb{Z}) \oplus (H_0({\partial}_1 F_{C_0};\mathbb{Z}) \otimes \mathbb{Z})$
 and $H_2(f^{-1}(N(C_0));\mathbb{Z}) \cong H_2(C_0 \times F_{C_0};\mathbb{Z}) \cong \mathbb{Z} \otimes H_0(F_{C_0};\mathbb{Z})$. The
 homomorphism from $H_0({\partial}_1 F_{C_0};\mathbb{Z})$ into $H_0(F_{C_0};\mathbb{Z})$ induced by the natural inclusion is
 assumed to be surjective. Thus, by the discussion here, the kernel of the homomorphism from $H^2(f^{-1}(N(C_0));\mathbb{Z})$
 into $H^2(f^{-1}(C_1);\mathbb{Z})$ induced by the
 natural inclusion is $\{0\}$. Thus, if the
 restriction of a bundle over $M$ to $f^{-1}(C_1)$ is trivial, then the restriction of the original bundle to $f^{-1}(N(C_0))$ is orientable by a discussion similar to that of the proof of the first case and it is a trivial bundle. \\ 
\ \ \ Consider a $2$-dimensional sphere ${C_s}^{\prime}$ in a closure of a connected component of ${\mathbb{R}}^n-f(S(f))$
 represented by $s \in S$ such that for a connected component $C_s$ of $f(S(f))$ and a small closed
 tubular neighborhood $N(C_s)$ defined in the assumption of this theorem, ${C_s}^{\prime}$ is a connected component of $\partial N(C_s)$.
 $f^{-1}({C_s}^{\prime})$ is diffeomorphic to ${C_s}^{\prime} \times F_s$ and $C_s \times F_s$. $H^1(f^{-1}({C_s}^{\prime});\mathbb{Z}/2\mathbb{Z}) \cong H^1(C_s \times F_s;\mathbb{Z}/2\mathbb{Z}) \cong H^1(F_s;\mathbb{Z}/2\mathbb{Z})$
 holds. $H^2(f^{-1}(N(C));\mathbb{Z}) \cong H^2(C \times F_C;\mathbb{Z}) \cong H^2(C;\mathbb{Z}) \otimes H^0(F_C;\mathbb{Z}) \cong \mathbb{Z} \otimes H^0(F_C;\mathbb{Z})$ and
 $H^2(f^{-1}({C_s}^{\prime});\mathbb{Z}) \cong H^2(C_s \times F_s;\mathbb{Z}) \cong (H^2(C_s;\mathbb{Z}) \otimes H^0(F_s;\mathbb{Z})) \oplus (\mathbb{Z} \otimes H^2(F_s;\mathbb{Z}))$ also hold.
 Thus, by the structures of these cohomology groups, if the 1st Stiefel-Whitney class of an $S^1$-bundle over $M$ vanishes on $F_s$ and the restriction of the bundle
 to the image of a section of the trivial bundle given by $f {\mid}_{f^{-1}({C_s}^{\prime})}:f^{-1}({C_s}^{\prime}) \rightarrow {C_s}^{\prime}$ is trivial, then the restriction
 of the $S^1$-bundle over $M$ to $f^{-1}({C_s}^{\prime})$ is trivial. Thus, by the discussion in the previous paragraph together with an induction based on the assumptions that $S$ is
 a starting set of the $(H_1,\mathbb{Z}/2\mathbb{}Z)$-graph of $f_F=(f,\{F_C\}_{C \in {\pi}_0 (f(S(f)))})$ and that the graph and the
$(H_0,R)$-graph of $f_F=(f,\{F_C\}_{C \in {\pi}_0 (f(S(f)))})$ are start-equivalent on the set $S$, for any
 connected component $C$ of $f(S(f))$, if we restrict any $S^1$-bundle over $M$ whose
 1st Stiefel-Whitney class vanishes on $i_{\ast}(H_1(F_s;\mathbb{Z}/2\mathbb{Z})) \subset H_1(M;\mathbb{Z}/2\mathbb{Z})$ ($s \in S$) and whose restriction of the bundle
 to the image of a section of the trivial bundle given by $f {\mid}_{f^{-1}({C_s}^{\prime})}:f^{-1}({C_s}^{\prime}) \rightarrow {C_s}^{\prime}$ is trivial to $f^{-1}(N(C))$, then the resulting bundle is
 orientable and trivial. This completes the proof of the second case. \\ 
\ \ \ We prove the statement in the last case. The homomorphism
 from $H_1(f^{-1}(C_1);\mathbb{Z}/2\mathbb{Z}) \cong H_1({\partial}_1 F_{C_0};\mathbb{Z}/2\mathbb{Z}) \oplus (H_1(C_0;\mathbb{Z}/2\mathbb{Z}) \otimes H_0({\partial}_1 F_{C_0};\mathbb{Z}/2\mathbb{Z}))$ into
 $H_1(f^{-1}(N(C_0));\mathbb{Z}/2\mathbb{Z}) \cong H_1(F_{C_0};\mathbb{Z}/2\mathbb{Z}) \oplus (H_1(C_0;\mathbb{Z}/2\mathbb{Z}) \otimes H_0(F_{C_0};\mathbb{Z}/2\mathbb{Z}))$ induced by
 the natural inclusion is surjective, since the homomorphism
 from $H_k({\partial}_1 F_{C_0};\mathbb{Z}/2\mathbb{Z})$ into $H_k(F_{C_0};\mathbb{Z}/2\mathbb{Z})$ induced by
 the natural inclusion is assumed to be surjective for $k=0,1$. It follows that the kernel of the homomorphism from $H^1(f^{-1}(N(C_0));\mathbb{Z}/2\mathbb{Z})$ into $H^1(f^{-1}(C_1);\mathbb{Z}/2\mathbb{Z})$ induced
 by the natural inclusion is $\{0\}$. \\
\ \ \ The homomorphism
 from $H_k({\partial}_1 F_{C_0};\mathbb{Z})$ into $H_k(F_{C_0};\mathbb{Z})$ and the homomorphism from $H_k({\partial}_1 F_{C_0};\mathbb{Z})$ into $H_k(F_{C_0};\mathbb{Z})$ induced by the natural inclusions
 are assumed to be surjective for $k=0,1$. Thus, $H_2(f^{-1}(N(C_0));\mathbb{Z}) \cong (H_1(C_0;\mathbb{Z}) \otimes H_1(F_{C_0};\mathbb{Z})) \oplus H_2(F_{C_0};\mathbb{Z})$ is regarded also as the direct sum of the following two groups.
\begin{enumerate}
\item The image of the restriction of the homomorphism from $H_2(f^{-1}(C_1);\mathbb{Z}) \cong (H_1(C_0;\mathbb{Z}) \otimes H_1({\partial}_1 F_{C_0};\mathbb{Z})) \oplus H_2({\partial}_1 F_{C_0};\mathbb{Z})$ into $H_2(f^{-1}(N(C_0));\mathbb{Z})$ induced by
 the natural inclusion to the part $H_1(C_0;\mathbb{Z}) \otimes H_1({\partial}_1 F_{C_0};\mathbb{Z})$ .
\item $H_2(F_{C_0};\mathbb{Z})$.
\end{enumerate}
 Hence, the kernel of the homomorphism from $H^2(f^{-1}(N(C_0));\mathbb{Z}) \cong (H^1(C_0;\mathbb{Z}) \otimes H^1(F_{C_0};\mathbb{Z})) \oplus H^2(F_{C_0};\mathbb{Z})$
 into $H^2(f^{-1}(C_1);\mathbb{Z}) \cong (H^1(C_0;\mathbb{Z}) \otimes H^1({\partial}_1 F_{C_0};\mathbb{Z})) \oplus H^2({\partial}_1 F_{C_0};\mathbb{Z})$ induced by
 the natural inclusion is $\{0\}$. \\
\ \ \ By the arguments before, if the restriction of
 an $S^1$-bundle over $M$ to $f^{-1}(C_1)$ is orientable and the Euler class of the (oriented) bundle obtained by the restriction
 vanishes, then the restriction of
 the $S^1$-bundle over $M$ to $f^{-1}(N(C_0))$ is orientable and the Euler class of
 the (oriented) bundle obtained by the restriction vanishes. \\
\ \ \ Consider a circle ${C_s}^{\prime}$ in a closure of a connected component of ${\mathbb{R}}^2-f(S(f))$
 represented by $s \in S$ such that for a connected component $C_s$ of $f(S(f))$ and a small closed
 tubular neighborhood $N(C_s)$ as in Definition \ref{def:2}, ${C_s}^{\prime}$ is a connected component of $\partial N(C)$.
 $f^{-1}({C_s}^{\prime})$ is diffeomorphic to ${C_s}^{\prime} \times F_s$ and $C_s \times F_s$. We
 have
 $H^1(f^{-1}({C_s}^{\prime});\mathbb{Z}/2\mathbb{Z}) \cong H^1(C_s \times F_s;\mathbb{Z}/2\mathbb{Z}) \cong \mathbb{Z}/2\mathbb{Z} \otimes H^1(F_s;\mathbb{Z}/2\mathbb{Z})$, $H^2(f^{-1}(N(C));\mathbb{Z}) \cong H^2(C \times F_C;\mathbb{Z}) \cong H^1(C;\mathbb{Z}) \otimes H^1(F_C;\mathbb{Z}) \cong \mathbb{Z} \otimes H^1(F_C;\mathbb{Z})$ and
 $H^2(f^{-1}({C_s}^{\prime});\mathbb{Z}) \cong H^2(C_s \times F_s;\mathbb{Z}) \cong (H^1(C_s;\mathbb{Z}) \otimes H^1(F_s;\mathbb{Z})) \oplus (H^0(C_s;\mathbb{Z}) \otimes H^2(F_s;\mathbb{Z})) \cong (\mathbb{Z} \otimes H^1(F_s;\mathbb{Z})) \oplus (\mathbb{Z} \otimes H^2(F_s;\mathbb{Z}))$. $S$ is
 a starting set of the $(H_1,\mathbb{Z}/2\mathbb{Z})$-graph of $f_F=(f,\{F_C\}_{C \in {\pi}_0 (f(S(f)))})$. This graph and the $(H_0,R)$-graph of $f_F=(f,\{F_C\}_{C \in {\pi}_0 (f(S(f)))})$ are start-equivalent on the set $S$. This graph and the
$(H_1,\mathbb{Z})$-graph of $f_F=(f,\{F_C\}_{C \in {\pi}_0 (f(S(f)))})$ are also start-equivalent on the set $S$. Thus, by the discussion in the previous paragraph, for any
 connected component $C$ of $f(S(f))$, if we restrict any $S^1$-bundle over $M$ whose restriction of the bundle
 to the total space of the trivial bundle given by $f {\mid}_{f^{-1}({C_s}^{\prime})}:f^{-1}({C_s}^{\prime}) \rightarrow {C_s}^{\prime}$ is trivial to $f^{-1}(N(C))$, then the resulting bundle is
 orientable and trivial. This completes the proof of the last case. \\ 
\ \ \ This completes the proof of all the statements. 
\end{proof}

  We show Theorem \ref{thm:1} in the case where the manifold $M$ is connected and $n=2,3$

\begin{Thm}
\label{thm:2}
Let $M$ be a closed and connected manifold of
 dimension $m$ and $f:M \rightarrow {\mathbb{R}}^n$ {\rm (}$m \geq n \geq 2${\rm )} be a $C^{\infty}$ L-trivial round fold map. For any connected component $C$ of $f(S(f))$, we denote by $N(C)$ a small closed tubular neighborhood of $C$ such that the
 inverse image has the structures of trivial bundles as in Definition \ref{def:2} and by $F_C$ a normal
 fiber of $C$ corresponding to the bundles $f^{-1}(N(C))$ over $C$. Let $F_C$ satisfy
 $H^{2}(F_C;\mathbb{Z}) \cong \{0\}$. Let $S$ be a starting set of the
$(H_1,\mathbb{Z}/2\mathbb{Z})$-graph of $f_F=(f,\{F_C\}_{C \in {\pi}_0 (f(S(f)))})$. Let $F_s$ be the regular fiber of a point in a connected component of ${\mathbb{R}}^n-f(S(f))$ whose
 closure is represented as a vertex $s \in S$ of the graph above and let $i_s:F_s \rightarrow M$ be the natural inclusion. Then, we have the followings. 
\begin{enumerate}
\item Let $n=3$ hold. Then, by a P-operation to $f$, for any closed manifold $M^{\prime}$ having the structure of an $S^1$-bundle over $M$ such that
 the followings hold, we
 can obtain a round fold map $f^{\prime}:M^{\prime} \rightarrow {\mathbb{R}}^3$. 
\begin{enumerate} 
\item The 1st
 Stiefel-Whitney class vanishes on ${i_s}_{\ast}(H_1(F_s;\mathbb{Z}/2\mathbb{Z})) \subset H_1(M;\mathbb{Z}/2\mathbb{Z})$ for any $s \in S$.
\item For a $2$-dimensional sphere ${C_0}^{\prime}$ in a connected component of $({\mathbb{R}}^3-f(S(f))) \bigcap f(M)$ such that for a
 connected component $C_0$ of $f(S(f))$ and the small closed
 tubular neighborhood $N(C_0)$ defined before, ${C_0}^{\prime}$ is a connected component of $\partial N(C_0)$, the restriction of the bundle
 to the image of a section of the trivial bundle given by $f {\mid}_{f^{-1}({C_0}^{\prime})}:f^{-1}({C_0}^{\prime}) \rightarrow {C_0}^{\prime}$ is trivial. 
\end{enumerate} 
\item Let $n=2$ hold. Furthermore, we assume that the $(H_1,\mathbb{Z}/2\mathbb{Z})$-graph of $f_F=(f,\{F_C\}_{C \in {\pi}_0 (f(S(f)))})$ and the $(H_1,\mathbb{Z})$-graph of $f_F=(f,\{F_C\}_{C \in {\pi}_0 (f(S(f)))})$ are start-equivalent on the set $S$. \\ 
\ \ \ For $s \in S$, let ${C_s}^{\prime}$ be a circle in a closure of a connected component of ${\mathbb{R}}^2-f(S(f))$ represented by $s \in S$ such that for a connected component $C_s$ of $f(S(f))$ and the small closed
 tubular neighborhood $N(C_s)$ defined before, ${C_s}^{\prime}$ is a connected component of $\partial N(C_s)$. 
Let $M^{\prime}$ be a manifold having the structure of an oriented $S^1$-bundle over $M$ such that
 the restriction of the bundle
 to the total space of the trivial bundle given by $f {\mid}_{f^{-1}({C_s}^{\prime})}:f^{-1}({C_s}^{\prime}) \rightarrow {C_s}^{\prime}$ is orientable and trivial. \\
\ \ \ Then, by a P-operation to $f$, for
 any closed manifold $M^{\prime}$ having the structure of an $S^1$-bundle over $M$, we
 can obtain a round fold map $f^{\prime}:M^{\prime} \rightarrow {\mathbb{R}}^2$.  
\end{enumerate}
\end{Thm}
\begin{proof}
We prove the first part. For any connected component $C$ of $f(S(f))$ and the small tubular neighborhood $N(C)$ of $C$, we
 have $H^2(f^{-1}(N(C));\mathbb{Z}) \cong H^2(F_C;\mathbb{Z}) \oplus (H^0(F_C;\mathbb{Z}) \otimes \mathbb{Z}) \cong H^0(F_C;\mathbb{Z}) \otimes \mathbb{Z}$. By an
 argument similar to that of Theorem \ref{thm:1}, the restriction of
 any $S^1$-bundle over $M$ such that the 1st
 Stiefel-Whitney class vanishes on ${i_s}_{\ast}(H_1(F_s;\mathbb{Z}/2\mathbb{Z})) \subset H_1(M;\mathbb{Z}/2\mathbb{Z})$ for any $s \in S$ to $f^{-1}(N(C))$ is orientable. Thus, by a discussion similar to that of the proof of Theorem \ref{thm:1}, we have only to show that
 for any connected component $C$ of $f(S(f))$ and each connected component $C^{\prime}$ of the boundary $\partial N(C)$ of $N(C)$, the restriction of the bundle
 over $M$ satisfying the conditions mentioned in the assumption of this theorem to $f^{-1}(C^{\prime})$ is trivial. Since the bundle $f^{-1}(N(C))$ over $C$ as in Definition \ref{def:2} is
 trivial and represented by $C \times F_C$, the restriction of the bundle over $M$ to $f^{-1}({C_0}^{\prime})$ is trivial by the second condition and $M$ is connected, it immediately follows. \\
\ \ \ We prove the second part. Let $C$ be a connected component of $f(S(f))$. For the small tubular neighborhood $N(C)$ of $C$, let $C^{\prime}$ be a connected component of the boundary of the tubular neighborhood $N(C)$ and let ${F_C}^{\prime}$ be
 a fiber of the bundle given by $f {\mid}_{f^{-1}(C^{\prime})}:f^{-1}(C^{\prime}) \rightarrow C^{\prime}$. We
 have $H^1(f^{-1}(C^{\prime});\mathbb{Z}/2\mathbb{Z}) \cong H^1({C}^{\prime} \times {F_C}^{\prime};\mathbb{Z}/2\mathbb{Z}) \cong \mathbb{Z}/2\mathbb{Z} \otimes H^1({F_C}^{\prime};\mathbb{Z}/2\mathbb{Z})$ and $H^1(f^{-1}(N(C));\mathbb{Z}/2\mathbb{Z}) \cong (H^1(C;\mathbb{Z}/2\mathbb{Z}) \otimes H^0(F_C;\mathbb{Z}/2\mathbb{Z})) \oplus ((H^0(C;\mathbb{Z}/2\mathbb{Z}) \otimes H^1(F_C;\mathbb{Z}/2\mathbb{Z})) \cong (\mathbb{Z}/2\mathbb{Z} \otimes H^0(F_C;\mathbb{Z}/2\mathbb{Z})) \oplus (\mathbb{Z}/2\mathbb{Z} \otimes H^1(F_C;\mathbb{Z}/2\mathbb{Z}))$. We
 also have $H^2(f^{-1}({C}^{\prime});\mathbb{Z}) \cong H^2({C}^{\prime} \times {F_C}^{\prime};\mathbb{Z}) \cong (H^1({C}^{\prime};\mathbb{Z}) \otimes H^1({F_C}^{\prime};\mathbb{Z})) \oplus (H^0({C}^{\prime};\mathbb{Z}) \otimes H^2({F_C}^{\prime};\mathbb{Z})) \cong (\mathbb{Z} \otimes H^1({F_C}^{\prime};\mathbb{Z})) \oplus (\mathbb{Z} \otimes H^2({F_C}^{\prime};\mathbb{Z}))$
 and $H^2(f^{-1}(N(C));\mathbb{Z}) \cong H^2(C \times F_C;\mathbb{Z}) \cong H^1(C;\mathbb{Z}) \otimes H^1(F_C;\mathbb{Z}) \cong \mathbb{Z} \otimes H^1(F_C;\mathbb{Z})$.
For any connected component $C$ of $f(S(f))$, the restriction of the bundle
 over $M$ satisfying the conditions mentioned in the assumption of this theorem to a fiber $F_C$ of the bundle $f^{-1}(N(C))$ over $C$ is orientable by the
 structures of the groups $H^1(f^{-1}(N(C));\mathbb{Z}/2\mathbb{Z})$ and $H^1(f^{-1}({C}^{\prime});\mathbb{Z}/2\mathbb{Z})$ together with the assumption that the $(H_1,\mathbb{Z}/2\mathbb{Z})$-graph of $f_F=(f,\{F_C\}_{C \in {\pi}_0 (f(S(f)))})$
 has a starting set $S$. Furthermore, the restriction of the bundle
 over $M$ satisfying the conditions mentioned in the assumption of this theorem to $f^{-1}(N(C))$ is orientable by the assumptions that the restriction of the bundle
 to the total space of the trivial bundle given by $f {\mid}_{f^{-1}({C_s}^{\prime})}:f^{-1}({C_s}^{\prime}) \rightarrow {C_s}^{\prime}$ is orientable and that the manifold $M$ is connected. Thus, by a discussion similar to that of the proof of Theorem \ref{thm:1}, for any connected component $C$ of $f(S(f))$, the restriction of the bundle to $f^{-1}(N(C))$ is trivial. This completes the proof of the desired fact.
\end{proof}

We immediately have the following corollary to Theorem \ref{thm:2}.

\begin{Cor}
\label{cor:1}
Let $M$ be a closed and connected manifold of
 dimension $m \geq 3$ and $f:M \rightarrow {\mathbb{R}}^3$ be a $C^{\infty}$ L-trivial round fold map such
 that $f(M)$ is diffeomorphic to $D^3$. For any connected component $C$ of $f(S(f))$, we denote by $N(C)$  a small closed tubular neighborhood of $C$ such that the
 inverse image has the structures of trivial bundles as in Definition \ref{def:2} and by $F_C$ the normal
 fiber of $C$ corresponding to the bundles $f^{-1}(N(C))$ over $C$. Let $F_C$ satisfy
 $H^{2}(F_C;\mathbb{Z}) \cong \{0\}$. Let $S$ be a starting set of the
$(H_1,\mathbb{Z}/2\mathbb{Z})$-graph of $f_F=(f,\{F_C\}_{C \in {\pi}_0 (f(S(f)))})$. Let $F_s$ be the regular fiber of a point in a connected component of ${\mathbb{R}}^n-f(S(f))$ whose
 closure is represented as a vertex $s \in S$ of the graph above and let $i_s:F_s \rightarrow M$ be the natural inclusion. \\
\ \ \ Then, by a P-operation to $f$, for any closed manifold $M^{\prime}$ having the structure of an $S^1$-bundle over $M$ such
 that the 1st
 Stiefel-Whitney class vanishes on ${i_s}_{\ast}(H_1(F_s;\mathbb{Z}/2\mathbb{Z})) \subset H_1(M;\mathbb{Z}/2\mathbb{Z})$ for any $s \in S$, we
 can obtain a round fold map $f^{\prime}:M^{\prime} \rightarrow {\mathbb{R}}^3$. 
\end{Cor}

We review the following propositions, which have been shown in \cite{kitazawa 5} and which are regarded as specific cases
 of Theorems \ref{thm:1} and \ref{thm:2}.

\begin{Prop}
\label{prop:3}
Let $M$ be a closed manifold of dimension $m$, $f:M \rightarrow {\mathbb{R}}^n$ be a $C^{\infty}$ L-trivial round fold map and $m \geq n \geq 3$. Assume that either of the followings holds.
\begin{enumerate}
\item $n \geq 4$ holds.
\item $n=3$ holds, $f(M)$ is diffeomorphic to $D^3$ and $M$ is connected.
\end{enumerate}
 Assume that for any connected component $C$ of $f(S(f))$, a normal fiber $F_C$ of $C$ corresponding to a bundle satisfies $H^{2}(F_C;\mathbb{Z}) \cong \{0\}$. Then
 for any closed manifold having the structure of an orientable $S^1$-bundle over $M$, by a P-operation by $S^1$ on $f$, we
 can obtain a round fold map $f^{\prime}:M^{\prime} \rightarrow {\mathbb{R}}^n$. \\ 
\ \ \ Furthermore, for any connected component $C$ of $f(S(f))$, let the
 normal fiber $F_C$ of $C$ satisfy $H^{1}(F_C;{\mathbb{Z}}_2) \cong \{0\}$. Then for any closed manifold having the structure of an $S^1$-bundle over $M$, by a P-operation by $S^1$ on $f$, we
 can obtain a $C^{\infty}$ L-trivial round fold map $f^{\prime}:M^{\prime} \rightarrow {\mathbb{R}}^n$.
\end{Prop}

\begin{Prop}
\label{prop:4}
Let $M$ be a closed manifold of dimension $m \geq 2$ and $f:M \rightarrow {\mathbb{R}}^2$ be a 
$C^{\infty}$ L-trivial round fold map.
\begin{enumerate}
\item For any connected component $C$ of $f(S(f))$, we denote by $N(C)$ a small closed tubular neighborhood of $C$ such that $f^{-1}(N(C))$ has the
 structures of trivial bundles
 as in Definition \ref{def:3} and  by $F_C$ a normal fiber of $C$ corresponding to the bundles $f^{-1}(N(C))$
 over $C$. Let $F_C$ satisfy $H^{1}(F_C;\mathbb{Z}) \cong H^{2}(F_C;\mathbb{Z}) \cong \{0\}$. Then, by a P-operation by $S^1$ on $f$, for any closed manifold having the structure of an $S^1$-bundle over $M$ such that
 for any connected component $C$ of $f(S(f))$, the restriction of the bundle to $f^{-1}(\partial N(C))$ is an orientable bundle, we
 can obtain a round fold map $f^{\prime}:M^{\prime} \rightarrow {\mathbb{R}}^2$. \\
\item Furthermore, if $f(M)$ is diffeomorphic
 to $D^2$ and $M$ is connected, then by a P-operation by $S^1$ on $f$, for any closed manifold having the structure of an $S^1$-bundle over $M$, we
 can obtain a round fold map $f^{\prime}:M^{\prime} \rightarrow {\mathbb{R}}^2$.
\end{enumerate}
\end{Prop}

\begin{Ex}
\label{ex:2}
 For any pair of positive integers $m$ and $n$ satisfying $m>n \geq 2$, we can consider round fold maps from $m$-dimensional closed manifolds into ${\mathbb{R}}^{n}$ as in Example \ref{ex:1} (\ref{ex:1.3}). A special generic map from
 the $m$-dimensional homotopy sphere into ${\mathbb{R}}^{n}$ presented in Example \ref{ex:1} (\ref{ex:1.1}) are simplest examples of such
 maps. If for the two integers, $m-n>2$ holds, then the maps satisfy the assumption of Theorem \ref{thm:1}. In this case, the maps satisfy the
 assumption of Proposition \ref{prop:3} or \ref{prop:4}. In this
 case, the
$(H_1;\mathbb{Z}/2\mathbb{Z})$-graph of $f_F=(f,\{F_C\}_{C \in {\pi}_0 (f(S(f)))})$ is a graph satisfying
 the followings. 
\begin{enumerate}
\item The set $E^{(H_1,\mathbb{Z}/2\mathbb{Z})}_{f_F} \subset V_{f_F} \times V_{f_F}$ consist of all the pairs of $(s_1,s_2)$ such that the closures
 of connected components of ${\mathbb{R}}^n-f(S(f))$ represented by $s_1$ and $s_2$ intersect.
\item The set consisting of just one vertex $s \in V_{f_F}$ representing the closure of the connected component of ${\mathbb{R}}^n-f(S(f))$ including a proper
 core of $f$ is a starting set of the graph.
\end{enumerate}

For any non-negative integer $k \neq 0,m-n$ and any commutative ring $R$, the $(H_k;R)$-graph of $f_F=(f,\{F_C\}_{C \in {\pi}_0 (f(S(f)))})$ and the graph above
 coincide. The $(H_0;R)$-graph of $f_F=(f,\{F_C\}_{C \in {\pi}_0 (f(S(f)))})$ satisfies the followings.
\begin{enumerate}
\item The set $E^{(H_0,R)}_{f_F} \subset V_{f_F} \times V_{f_F}$ consist of all the pairs of $(s_1,s_2)$ such that the closures
 of connected components of ${\mathbb{R}}^n-f(S(f))$ represented by $s_1$ and $s_2$ intersect but a pair $(s_1,s_2)$ where
 $s_1$ or $s_2$ represents the closure of the unbounded connected component of ${\mathbb{R}}^n-f(S(f))$.
\item The set consisting of just one vertex $s \in V_{f_F}$ representing the closure of the connected component of ${\mathbb{R}}^n-f(S(f))$ including a proper
 core of $f$ is a starting set of the graph.
\end{enumerate}

\ \ \ If $m-n=1$ holds, then
 the maps do not satisfy the assumption of Propositions \ref{prop:3} or \ref{prop:4}, either. However, the maps satisfy the assumptions
 of Theorems \ref{thm:1} and \ref{thm:2}. In this
 case, for any commutative ring $R$, the
$(H_1,R)$-graph of $f_F=(f,\{F_C\}_{C \in {\pi}_0 (f(S(f)))})$ is a graph satisfying
 the followings. 
\begin{enumerate}
\item The set ${E}^{(H_1,R)}_{f_F} \subset V_{f_F} \times V_{f_F}$ consist of the followings.
\begin{enumerate}
\item All the pairs $(s_1,s_2) \in V_{f_F} \times V_{f_F}$ such that the closures
 of connected components of ${\mathbb{R}}^n-f(S(f))$ represented by $s_1$ and $s_2$ intersect and that the closure of the connected component of ${\mathbb{R}}^n-f(S(f))$ represented by $s_1$
 is in the bounded connected component of the complement set of the connected component of ${\mathbb{R}}^n-f(S(f))$ whose closure is represented by $s_2$ in ${\mathbb{R}}^n$.
\item The pair $(s_1,s_2) \in V_{f_F} \times V_{f_F}$ such that $s_1$ represents the closure of the unbounded connected component of ${\mathbb{R}}^n-f(S(f))$ and that
 the previous closed set and the closure of the connected component of ${\mathbb{R}}^n-f(S(f))$ represented by $s_2$ intersect.
\end{enumerate}
\item The set consisting of just one vertex $s \in V_{f_F}$ representing the closure of the connected component of ${\mathbb{R}}^n-f(S(f))$ including a proper
 core of $f$ is a starting set of the graph.
\end{enumerate} 

For any non-negative integer $k \neq 1$ and any commutative ring $R$, the $(H_k;R)$-graph of $f_F=(f,\{F_C\}_{C \in {\pi}_0 (f(S(f)))})$ is a graph mentined in the case where $m-n>2$ holds. \\
\end{Ex}

By analogy of the proof of Theorem \ref{thm:1}, we have the following theorem.

\begin{Thm}
\label{thm:3}
Let $M$ be a closed manifold of dimension $m$ and $f:M \rightarrow {\mathbb{R}}^n$
 {\rm (}$m \geq n \geq 2${\rm )} be a $C^{\infty}$ L-trivial round fold map. For any connected component $C$ of $f(S(f))$, we denote by $N(C)$ a small closed tubular neighborhood of $C$ such that the
 inverse image has the structures of trivial bundles as in Definition \ref{def:2} and by $F_C$ a normal
 fiber of $C$ corresponding to the bundles $f^{-1}(N(C))$ over $C$. Let $H_1(F_C;\mathbb{Z})$ be torsion-free. 
Let $S$ be a starting set of the
$(H_2,\mathbb{Z})$-graph of $f_F=(f,\{F_C\}_{C \in {\pi}_0 (f(S(f)))})$. Let $F_s$ be the regular fiber of a point in a connected component of ${\mathbb{R}}^n-f(S(f))$ whose
 closure is represented as a vertex $s \in S$ of the graph above and let $i_s:F_s \rightarrow M$ be the natural inclusion. Then, we have the followings.
\begin{enumerate}
\item Let $n \geq 4$ hold. Then, by a P-operation on $f$, for any
 closed manifold $M^{\prime}$ having the structure of an oriented $S^1$-bundle over $M$ whose
 Euler class vanishes on ${i_s}_{\ast}(H_2(F_s;\mathbb{Z})) \subset H_2(M;\mathbb{Z})$ for any $s \in S$, we
 can obtain a round fold map $f^{\prime}:M^{\prime} \rightarrow {\mathbb{R}}^n$. 
\item Let $n=3$ hold. Furthermore, for any commutative ring $R$, the $(H_2,\mathbb{Z})$-graph of $f_F=(f,\{F_C\}_{C \in {\pi}_0 (f(S(f)))})$
 and the $(H_0,R)$-graph of $f_F=(f,\{F_C\}_{C \in {\pi}_0 (f(S(f)))})$ are start-equivalent on the set $S$. 
Then, by a P-operation on $f$, for any closed manifold $M^{\prime}$ having the structure of an oriented $S^1$-bundle over $M$ such that
 the followings hold, we
 can obtain a round fold map $f^{\prime}:M^{\prime} \rightarrow {\mathbb{R}}^n$. 
\begin{enumerate} 
\item The Euler class vanishes on ${i_s}_{\ast}(H_2(F_s;\mathbb{Z})) \subset H_2(M;\mathbb{Z})$ for any $s \in S$.
\item For $s \in S$, let ${C_s}^{\prime}$ be a 2-dimensional sphere in a closure of a connected component of ${\mathbb{R}}^n-f(S(f))$ represented by $s \in S$ such that for a connected component $C_s$ of $f(S(f))$ and the small closed
 tubular neighborhood $N(C_s)$ defined before, ${C_s}^{\prime}$ is a connected component of $\partial N(C)$. The restriction of the bundle
 to the image of a section of the trivial bundle given by $f {\mid}_{f^{-1}({C_s}^{\prime})}:f^{-1}({C_s}^{\prime}) \rightarrow {C_s}^{\prime}$ is trivial. 
\end{enumerate} 
\item Let $n=2$ hold. Furthermore, we also assume that for any commutative ring $R$, the $(H_2;\mathbb{Z})$-graph
 of $f_F=(f,\{F_C\}_{C \in {\pi}_0 (f(S(f)))})$
 and the $(H_1;\mathbb{Z})$-graph of $f_F=(f,\{F_C\}_{C \in {\pi}_0 (f(S(f)))})$ are start-equivalent on the set $S$. \\
\ \ \ For $s \in S$, let ${C_s}^{\prime}$ be a circle in a closure of a connected component of ${\mathbb{R}}^2-f(S(f))$ represented by $s \in S$ such that for a connected component $C_s$ of $f(S(f))$ and the small closed
 tubular neighborhood $N(C_s)$ defined before, ${C_s}^{\prime}$ is a connected component of $\partial N(C_s)$. Let $M^{\prime}$ be a manifold having the structure of an oriented $S^1$-bundle over $M$ such that
 the restriction of the bundle
 to the total space of the trivial bundle given by $f {\mid}_{f^{-1}({C_s}^{\prime})}:f^{-1}({C_s}^{\prime}) \rightarrow {C_s}^{\prime}$ is orientable and trivial. \\
\ \ \ Then, by a P-operation on $f$, we
 can obtain a round fold map $f^{\prime}:M^{\prime} \rightarrow {\mathbb{R}}^n$.
\end{enumerate}
\end{Thm}
\begin{proof}
Let $C_0$ be a connected component of $f(S(f))$ which is in the closures of two connected component of ${\mathbb{R}}^n-f(S(f))$ represented by
 two vertices $s_1,s_2 \in V_{f_F}$ and let $(s_1,s_2) \in E_{f_F}$. We use $C_j$ and ${\partial}_j F_{C_0}$ as in the proof of Theorem \ref{thm:1}. \\
\ \ \ The homomorphism
 from $H_2(f^{-1}(C_1);\mathbb{Z}) \cong H_2(C_1 \times {\partial}_1 F_{C_0};\mathbb{Z}) \cong (H_0({\partial}_1 F_{C_0};\mathbb{Z}) \otimes H_2(C_0;\mathbb{Z})) \oplus (H_1(C_0;\mathbb{Z}) \otimes H_1({\partial}_1 F_{C_0};\mathbb{Z})) \oplus H_2({\partial}_1 F_{C_0};\mathbb{Z})$
 into $H_2(f^{-1}(N(C));\mathbb{Z}) \cong (H_0(F_C;\mathbb{Z}) \otimes H_2(C;\mathbb{Z})) \oplus (H_1(C;\mathbb{Z}) \otimes H_1(F_C;\mathbb{Z})) \oplus H_2(F_C;\mathbb{Z})$ induced by
 the natural inclusion is assumed to be surjective in the first case, since the homomorphism
 from $H_2({\partial}_1 F_{C_0};\mathbb{Z})$ into $H_2(F_C;\mathbb{Z})$ induced by the natural inclusion is assumed to be surjective. In the second case, the homomorphism
 from $H_k({\partial}_1 F_{C_0};\mathbb{Z})$ into $H_k(F_C;\mathbb{Z})$ induced by the natural inclusion is assumed to be surjective for $k=0,2$ and as a result, the homomorphism above is surjective. In the last case, the homomorphism from $H_k({\partial}_1 F_{C_0};\mathbb{Z})$ into $H_k(F_C;\mathbb{Z})$ induced by the natural
 inclusion is assumed to be surjective for $k=1, 2$ and as a result, the homomorphism above is surjective. Thus, the
 kernel of the homomorphism from $H^2(f^{-1}(N(C));\mathbb{Z})$ into $H^2(f^{-1}(C_1);\mathbb{Z})$ induced by
 the natural inclusion is $\{0\}$ in all the three cases by virtue of the universal coefficient theorem together with the assumption that the
 group $H_1(F_C;\mathbb{Z})$ is torsion-free for any connected component $C$ of $f(S(f))$. So, if the Euler class of the restriction of
 an oriented $S^1$-bundle over $M$ to $f^{-1}(C_1)$ vanishes, then the Euler classes of the restrictions of the bundle
 over $M$ to $f^{-1}(N(C_0))$ vanish and the obtained bundles are trivial bundles. \\
\ \ \ By arguments similar to ones in the proofs of Theorem \ref{thm:1}, for each case, for any connected component $C$ of $f(S(f))$, the restriction of any oriented $S^1$-bundle
 over $M$ satisfying the given condition to $f^{-1}(N(C))$ is
 a trivial bundle. This completes the proof of the statement for all the three cases. 
\end{proof}

By analogy of the proof of Theorem \ref{thm:2}, we have the following theorem.

\begin{Thm}
\label{thm:4}
Let $M$ be a closed and connected manifold of
 dimension $m \geq 3$ and $f:M \rightarrow {\mathbb{R}}^3$ be a $C^{\infty}$ L-trivial round fold map. For any connected component $C$ of $f(S(f))$, we denote by $N(C)$ a small closed tubular neighborhood of $C$ such that the
 inverse image has the structures of trivial bundles as in Definition \ref{def:2} and by $F_C$ a normal
 fiber of $C$ corresponding to the bundles $f^{-1}(N(C))$ over $C$.  Let $H_1(F_C;\mathbb{Z})$ be torsion-free. 
Let $S$ be a starting set of the
$(H_2,\mathbb{Z})$-graph of $f_F=(f,\{F_C\}_{C \in {\pi}_0 (f(S(f)))})$. Let $F_s$ be the regular fiber of a point in a connected component of ${\mathbb{R}}^n-f(S(f))$ whose
 closure is represented as a vertex $s \in S$ of the graph above and let $i_s:F_s \rightarrow M$ be the natural inclusion. Then, by
 a P-operation to $f$, for any closed manifold $M^{\prime}$ having the structure of an oriented $S^1$-bundle over $M$ such that
 the followings hold, we
 can obtain a round fold map $f^{\prime}:M^{\prime} \rightarrow {\mathbb{R}}^3$. 
\begin{enumerate} 
\item The Euler class vanishes on ${i_s}_{\ast}(H_2(F_s;\mathbb{Z})) \subset H_2(M;\mathbb{Z})$ for any $s \in S$.
\item For a $2$-dimensional sphere ${C_0}^{\prime}$ in a connected component of $({\mathbb{R}}^3-f(S(f))) \bigcap f(M)$ such that for a
 connected component $C_0$ of $f(S(f))$ and the small closed
 tubular neighborhood defined before, ${C_0}^{\prime}$ is a connected component of $\partial N(C_0)$, the restriction of the bundle
 to the image of a section of the trivial bundle given by $f {\mid}_{f^{-1}({C_0}^{\prime})}:f^{-1}({C_0}^{\prime}) \rightarrow {C_0}^{\prime}$ is trivial.
\end{enumerate} 

\end{Thm}
\begin{proof}
Let $C$ be a connected component of $f(S(f))$. For the small tubular neighborhood $N(C)$ of $C$, let $C^{\prime}$ be a connected component of the boundary of the tubular neighborhood $N(C)$ and let ${F_C}^{\prime}$ be
 a fiber of the bundle given by $f {\mid}_{f^{-1}(C^{\prime})}:f^{-1}(C^{\prime}) \rightarrow C^{\prime}$. We have $H^2(f^{-1}({C}^{\prime});\mathbb{Z}) \cong H^2({C}^{\prime} \times {F_C}^{\prime};\mathbb{Z}) \cong (H^2({C}^{\prime};\mathbb{Z}) \otimes H^0({F_C}^{\prime};\mathbb{Z})) \oplus (H^0({C}^{\prime};\mathbb{Z}) \otimes H^2({F_C}^{\prime};\mathbb{Z})) \cong (\mathbb{Z} \otimes H^0({F_C}^{\prime};\mathbb{Z})) \oplus (\mathbb{Z} \otimes H^2({F_C}^{\prime};\mathbb{Z}))$
 and $H^2(f^{-1}(N(C));\mathbb{Z}) \cong H^2(C \times F_C;\mathbb{Z}) \cong (H^2(C;\mathbb{Z}) \otimes H^0(F_C;\mathbb{Z})) \oplus (H^0(C;\mathbb{Z}) \otimes H^2(F_C;\mathbb{Z})) \cong \mathbb{Z} \otimes H^0(F_C;\mathbb{Z}) \oplus (\mathbb{Z} \otimes H^2(F_C;\mathbb{Z}))$.
For any connected component $C$ of $f(S(f))$, the restriction of an oriented $S^1$-bundle
 over $M$ satisfying the conditions mentioned in the assumption of this theorem to a fiber $F_C$ of the bundle $f^{-1}(N(C))$ over $C$ is trivial by the
 structures of the groups $H^2(f^{-1}(N(C));\mathbb{Z})$ and $H^2(f^{-1}({C}^{\prime});\mathbb{Z})$ together with the assumption that the $(H_2,\mathbb{Z})$-graph of $f_F=(f,\{F_C\}_{C \in {\pi}_0 (f(S(f)))})$ has a starting set $S$. Furthermore, the restriction of the bundle
 over $M$ satisfying the conditions mentioned in the assumption of this theorem to $f^{-1}(N(C))$ is trivial by the assumptions that the restriction of the bundle
 to the total space of the trivial bundle given by $f {\mid}_{f^{-1}({C_0}^{\prime})}:f^{-1}({C_0}^{\prime}) \rightarrow {C_0}^{\prime}$ is trivial and that the manifold $M$ is connected. Thus, for any connected component $C$ of $f(S(f))$, the restriction of the bundle to $f^{-1}(N(C))$ is trivial. This completes the proof of the desired fact.
\end{proof}

\ \ \ We immediately have the following corollary to Theorem \ref{thm:4}.

\begin{Cor}
\label{cor:2}
Let $M$ be a closed and connected manifold of
 dimension $m \geq 3$ and $f:M \rightarrow {\mathbb{R}}^3$ be a $C^{\infty}$ L-trivial round fold map such that $f(M)$
 is diffeomorphic to $D^3$. For any connected component $C$ of $f(S(f))$, we denote by $N(C)$ a small closed tubular neighborhood of $C$ such that the
 inverse image has the structures of trivial bundles as in Definition \ref{def:2} and by $F_C$ the normal
 fiber of $C$ corresponding to the bundles $f^{-1}(N(C))$ over $C$.  Let $H_1(F_C;\mathbb{Z})$ be torsion-free. 
Let $S$ be a starting set of the
$(H_2,\mathbb{Z})$-graph of $f_F=(f,\{F_C\}_{C \in {\pi}_0 (f(S(f)))})$. Let $F_s$ be the regular fiber of a point in a connected component of ${\mathbb{R}}^n-f(S(f))$ whose
 closure is represented as a vertex $s \in S$ of the graph above and let $i_s:F_s \rightarrow M$ be the natural inclusion. \\
\ \ \ Then, by a P-operation on $f$, for any closed manifold $M^{\prime}$ having the structure of an oriented $S^1$-bundle over $M$ such that
 the Euler class vanishes on ${i_s}_{\ast}(H_2(F_s;\mathbb{Z})) \subset H_2(M;\mathbb{Z})$ for any $s \in S$, we
 can obtain a round fold map $f^{\prime}:M^{\prime} \rightarrow {\mathbb{R}}^3$. 
\end{Cor}

\begin{Ex}
\label{ex:3}
A round fold map $f$ presented in the beginning of Exmaple \ref{ex:2} satisfies the assumption of Theorems \ref{thm:3} and \ref{thm:4}. In the case where for the two integers, $m-n=2$ holds, the map
 satisfies the assumptions of Theorems \ref{thm:3} and \ref{thm:4} but does not satisfy the assumption of Theorem \ref{thm:1} or that of Theorem \ref{thm:2}, either. In this case, the map does not satisfy the
 assumption of Proposition \ref{prop:3} or Proposition \ref{prop:4}, either. 
\end{Ex}

\section{Algebraic and differential topological properties of round fold maps given by P-operations by $S^1$ and their source manifolds}
\label{sec:4}
We explicitly apply P-operations by $S^1$ to some round fold maps satisfying the assumption of Propositions, Theorems and Corollaries in section \ref{sec:3}. In this section, we
 use the following terminologies and notations. \\ 
\ \ \ For a freely generated commutative group $A$ of rank $r \geq 1$, $x \in A$ is said to
 be {\it primitive} if for the subgroup generated by $x$ (we denote the subgroup by $<x>$), the quotient group $A/<x>$ is a freely
 generated commutative group of rank $r-1$. \\
\ \ \ The diffeomorphism type of a manifold having the structure of a non-trivial $S^k$-bundle ($k=2,3$) over $S^2$ is unique. We
 denote the manifold by $S^k \tilde{\times} S^2$. \\
\subsection{Applications of Propositions and Theorems in section \ref{sec:3} to round fold maps in Example \ref{ex:1} (\ref{ex:1.1}) and (\ref{ex:1.3})}
We may apply Propositions, Theorems and Corollaries in section \ref{sec:3} to obtain a new
 round fold map by a P-operation by $S^1$ to a round fold map $f:M \rightarrow {\mathbb{R}}^n$ in Example \ref{ex:1} (\ref{ex:1.1}) (see also \cite{kitazawa 3}). In fact, by
 any P-operation by $S^1$ to the map $f$, we
 obtain the product of $S^1$ and the original source manifold as the resulting source manifold if the dimension of the original source manifold is larger than $2$. Moreover, by a P-operation by
$S^1$ to such a map from the $2$-dimensional standard sphere $S^2$ into the plane, we obtain a manifold having the structure of an $S^1$-bundle over $S^2$ as the
resulting source manifold and any manifold having the structure of an $S^1$-bundle over $S^2$ is realized as the resulting
source manifold by a P-operation by
 $S^1$ to every such map from $S^2$ into the plane. If for the round fold map $f:M \rightarrow {\mathbb{R}}^n$, the dimension of the source manifold $M$ and $n$ coincide, then the obtained maps
 are maps as in Example \ref{ex:1} (\ref{ex:1.2}) with $F:=S^1$. \\
\ \ \ We consider P-operations by $S^1$ to round fold maps $f:M \rightarrow {\mathbb{R}}^n$ ($n \geq 2$) in Example \ref{ex:1} (\ref{ex:1.3}). \\
\ \ \ First, in the case where $n \geq 3$ and $m \geq 2n$ hold in this example, the manifold having the structure of an $S^1$-bundle over
 the original source manifold $M$ is always the product of $S^1$ and $M$. By any P-operation by $S^1$, we always obtain a source
 manifold diffeomorphic to the product of the two manifolds. \\
\ \ \ In the case where $n=2$ and $m>2n=4$ hold in this example, on any manifold having the structure of an $S^1$-bundle over the original source manifold $M$, we can
 construct a round fold map by a P-operation by $S^1$ to the original map $f$ by Propositions \ref{prop:3} and \ref{prop:4} (we can apply Theorems \ref{thm:1}, \ref{thm:2}, \ref{thm:3} and \ref{thm:4} too). The author studied
 the diffeomorphism types of the resulting source manifolds of the resulting round
 fold maps in \cite{kitazawa 5}. \\
\ \ \ In the case where $(m,n)=(4,2)$ holds in this example, we cannot apply Proposition \ref{prop:4}. Moreover, we cannot apply Theorem \ref{thm:1} or \ref{thm:2}, either. On the other hand, we can apply Theorems \ref{thm:3} and \ref{thm:4}. 
 As a result on the diffeomorphism types of the source manifolds of round fold maps obtained by $P$-operations by $S^1$ to the original round fold map $f$, we have the following theorem. \\ 

\begin{Thm}
\label{thm:5}
Let $M$ be a connected sum of $l \in \mathbb{N}$ closed manifolds having the structure of an $S^2$-bundles over $S^2$. By a P-operation by $S^1$ to
 the map in Example \ref{ex:3}, for any cohomology class $\alpha \in H^2(M;\mathbb{Z})$ vanishing on all the cycles corresponding to the fibers of the bundles, we can construct
 a round fold map from any closed manifold $M(\alpha)$ having the
 structure of an oriented $S^1$-bundle over $M$ whose Euler class is $\alpha$ into ${\mathbb{R}}^2$. Furthermore, we have the followings.
\begin{enumerate}
\item $\alpha \in H^2(M;\mathbb{Z})$ is chosen to be primitive and if it is primitive, then $M(\alpha)$ is simply-connected.
\item If $M(\alpha)$ is simply-connected, then $\alpha \in H^2(M;\mathbb{Z})$ must be primitive and $M(\alpha)$ is diffeomorphic to a connected sum of $2l-2$ copies of $S^2 \times S^3$ {\rm (}we replace this by $S^5$ in the case where $l=1$ holds{\rm )} and a closed manifold having the
 structure of an $S^3$-bundle over $S^2$.
\item If $M$ is represented as a connected sum of finite copies of $S^2 \times S^2$ and $M(\alpha)$ is simply-connected, then $M(\alpha)$ is
 represented as a connected sum of $2l-1$ copies of $S^2 \times S^3$. 
\item If $M$ is represented as a connected sum of more than one closed manifolds having the structures of $S^2$-bundles over $S^2$ {\rm (}$l>1$ must hold{\rm )} and not
 represented as a connected sum of finite copies of $S^2 \times S^2$, then for any closed manifold represented as a connected sum of $2l-2$ copies of $S^2 \times S^3$ and a closed manifold having the
 structure of an $S^3$-bundle over $S^2$, we can choose $\alpha \in H^2(M;\mathbb{Z})$ so that $M(\alpha)$ is diffeomorphic to the manifold.   
\item If $M$ is a closed manifold having the structure of an $S^2$-bundle over $S^2$ and $M(\alpha)$ is
 simply-connected, then $M(\alpha)$ is diffeomorphic to $S^2 \times S^3$.
\end{enumerate}
\end{Thm}

\ \ \ For the proof of Theorem \ref{thm:5}, we need the following proposition, which is a main theorem of \cite{duan liang}. See also
 \cite{barden}, in which Barden classified closed and simply-connected $5$-dimensional manifolds completely. 

\begin{Prop}[\cite{duan liang}]
\label{prop:5}
Let $X$ be a closed and simply-connected manifold which may not be smooth of dimension $4$. Let $l>1$ be an integer and let $H^2(X;\mathbb{Z}) \cong {\mathbb{Z}}^{l}$ hold.
\begin{enumerate}
\item For any primitive element $\alpha \in H^2(X;\mathbb{Z})$, let $X(\alpha)$ be a manifold having the structure
 of an oriented $S^1$-bundle whose Euler class is $\alpha$. Then for any $\alpha$, $X(\alpha)$ is simply-connected and
 have a differentiable structure and one of the followings holds.
\begin{enumerate}
\item $X(\alpha)$ is diffeomorphic to a connected sum of $l-1$ copies of $S^2 \times S^3$  
\item $X(\alpha)$ is diffeomorphic to a connected sum of $l-2$ copies of $S^2 \times S^3$ {\rm (}we replace this by $S^5$ in the case where $l=2$ holds{\rm )} and $S^3 \tilde{\times}  S^2$.
\end{enumerate}  
\item Let $X$ be smooth manifold and let $\alpha \in H^2(X;\mathbb{Z})$ be primitive. A manifold $X(\alpha)$ having the structure
 of an oriented $S^1$-bundle whose Euler class is $\alpha$ is not diffeomorphic to
 a connected sum of $l-1$ copies of $S^2 \times S^3$ if and only if the followings hold.
\begin{enumerate}
\item The 2nd Stiefel-Whitney class $\beta$ of the tangent bundle $TX$ over $X$ does not vanish.
\item $\alpha$ is not equal to $\beta$ mod $2$. 
\end{enumerate}
\end{enumerate}
\end{Prop}

\begin{Ex}
\label{ex:4}
Let $X$ be a connected sum of a finite number of closed smooth manifolds having the structures of $S^2$-bundles over $S^2$. \\
\ \ \ Assume that all the bundles appeared in the connected sum are trivial. Then the 2nd
 Stiefel-Whitney class $\beta$ of the tangent bundle $TX$ over $X$ vanishes and $X(\alpha)$ in Proposition \ref{prop:5} is diffeomorphic to
 a connected sum of finite copies of $S^2 \times S^3$ by the proposition. \\  
\ \ \ Assume that there exist a bundle which is not trivial and a trivial bundle in the connected sum. Then the 2nd
 Stiefel-Whitney class $\beta$ of the tangent bundle $TX$ over $X$ vanishes on cycles corresponding to fibers of all the $S^2$-bundles in this family. On the
 other hand, the class does not vanish on the cycle corresponding to the base space of an $S^2$-bundle in the family which is not trivial. We can take $\alpha \in H^2(X;\mathbb{Z})$ so
 that $\alpha$ vanishes on cycles corresponding to fibers of all the $S^2$-bundles appeared in the connected sum and that the class does not vanish only on the cycle corresponding to the base
 space of just one trivial $S^2$-bundle in the family. In this case, $\alpha$ is not equal to $\beta$ mod $2$. Then $X(\alpha)$ in Proposition \ref{prop:5} is not diffeomorphic to
 a connected sum of finite copies of $S^2 \times S^3$ and is diffeomorphic to a connected sum of finite copies of $S^2 \times S^3$ and $S^3 \tilde{\times} S^2$ by the proposition. We can also take $\alpha \in H^2(X;\mathbb{Z})$ so
 that $\alpha$ is equal to $\beta$ mod $2$. Then, $X(\alpha)$ in Proposition \ref{prop:5} is diffeomorphic to
 a connected sum of finite copies of $S^2 \times S^3$.    
\end{Ex}

\begin{proof}[Proof of Theorem \ref{thm:5}]
A round fold map into ${\mathbb{R}}^2$ as in Example \ref{ex:1} (\ref{ex:1.3}) exists on $M$. We can apply Theorem \ref{thm:2} to this map. This
 completes the proof of the former part. We prove five statements in the latter part. \\
\ \ \ We can choose $\alpha \in H^2(M;\mathbb{Z})$ so that it vanishes on all the cycles corresponding to the fibers of the $C^{\infty}$ $S^2$-bundles over $S^2$ and on all but one cycles
 corresponding to the base spaces of the $C^{\infty}$ $S^2$-bundles over $S^2$ and that the value is $1$ on the cycle corresponding to the base space of the remained $C^{\infty}$ $S^2$-bundle over $S^2$. Then $\alpha$ is primitive and we obtain a round fold map
 from $M(\alpha)$ into ${\mathbb{R}}^2$ by a $P$-operation by $S^1$. By Proposition \ref{prop:5}, we have the first statement. \\
\ \ \ For the proof of the second statement, we review and apply the method used in the proof of Lemma 3 of \cite{duan liang}, which is a key lemma to Theorem 2 of \cite{duan liang} or Proposition \ref{prop:5} of the present paper. \\
\ \ \ From the homotopy sequence of the bundle $M(\alpha)$ over $M$, ${\pi}_1(M(\alpha))$ is cyclic. It follows that ${\pi}_1(M(\alpha)) \cong H_1(M(\alpha);\mathbb{Z})$.
Substituting $H^k(M;\mathbb{Z}) \cong \{0\}$ for any odd integer $k$ and any integer $k$ not smaller than $5$ in the Gysin sequence of
the bundle $M(\alpha)$ over $M$ yields the following two exact sequences where $p:M(\alpha) \rightarrow M$ is the projection of the bundle and where $\bigcup \alpha$ denotes
 the operation of taking cup products with $\alpha$.

$$
\begin{CD}
\{0\} @> p^{\ast}  >> H^3(M(\alpha);\mathbb{Z}) @> >>  H^2(M;\mathbb{Z}) @> \bigcup \alpha  >> \\ H^4(M;\mathbb{Z}) \cong \mathbb{Z} @> p^{\ast} >> H^4(M(\alpha);\mathbb{Z}) @> >> \{0\} \\
\end{CD}
$$

$$
\begin{CD}
\{0\} @> p^{\ast}  >> H^5(M(\alpha);\mathbb{Z}) @> >>  H^4(M;\mathbb{Z}) \cong \mathbb{Z} @> \bigcup \alpha  >> \{0\} 
\end{CD}
$$

\ \ \ By the second sequence, we have $H^5(M(\alpha);\mathbb{Z}) \cong \mathbb{Z}$. The homomorphism in the first sequence $\bigcup \alpha:H^2(M;\mathbb{Z}) \rightarrow H^4(M;\mathbb{Z})$ is
 surjective if and only if $\alpha$ is primitive. Furthermore, $\bigcup \alpha:H^2(M;\mathbb{Z}) \rightarrow H^4(M;\mathbb{Z})$ is surjective if
 and only if $H^4(M(\alpha);\mathbb{Z}) \cong H_1(M(\alpha);\mathbb{Z}) \cong {\pi}_1(M(\alpha)) \cong \{0\}$ holds (we can
 apply Poincare duality theorem since $H^5(M(\alpha);\mathbb{Z}) \cong \mathbb{Z}$ holds). Furthermore, by the first sequence, $H^3(M(\alpha);\mathbb{Z}) \cong {\mathbb{Z}}^{2l-1}$ holds if $\alpha$ is primitive. By virtue of Poincare duality
 theorem, $H^3(M(\alpha);\mathbb{Z}) \cong H_2(M(\alpha);\mathbb{Z})$ always holds. Thus, by the theory of \cite{barden} or Proposition \ref{prop:5}, the second statement
 of the five statements holds. \\
\ \ \ By the second statement of the five statements and Example \ref{ex:4} (or Proposition \ref{prop:5}), the third and fourth statements of the five statements hold. \\
\ \ \ The 2nd Stiefel-Whitney class of the tangent bundle of $S^2 \tilde{\times} S^2$ vanishes on the cycle corresponding to the fiber and does not vanish on one corresponding
 to the base space. So, if in the situation of the last statement of the five statements, $M$ is diffeomorphic to $S^2 \tilde{\times} S^2$ and $M(\alpha)$ is
 simply-connected, then $\alpha$ must be primitive and as a result it equals the 2nd Stiefel-Whitney class of $S^2 \tilde{\times} S^2$ mod $2$. Furthermore, by
 Proposition \ref{prop:5}, $M(\alpha)$ is diffeomorphic to $S^2 \times S^3$ under this condition and the same fact follows by the same proposition if $M=S^2 \times S^2$. Thus this
 completes the proof of the last one of the five statements. \\
\ \ \ This completes the proof.
\end{proof}

Note that also in \cite{kitazawa}, \cite{kitazawa 2} and \cite{kitazawa 4}, round fold maps from the simply-connected $5$-dimensional manifolds represented as connected sums of manifolds having the
 structures of $S^3$-bundles over $S^2$ into the plane are obtained. Regular fibers of
 these maps are disjoint unions of finite copies of $S^3$. Theorem \ref{thm:5} shown above gives new explicit examples of round fold maps
 from such simply-connected $5$-dimensional manifolds into the plane. 

\subsection{Applications of Propositions, Theorems and Corollaries mentioned in section \ref{sec:3} to round fold maps in Example \ref{ex:1} (\ref{ex:1.4})} 
We consider a round fold map $f$ from a closed $m$-dimensional manifold $M$ into ${\mathbb{R}}^n$ ($m \geq n \geq 2$) in Example \ref{ex:1} (\ref{ex:1.4}).

\subsubsection{The case where $n \geq 4$ holds.}
Let $n \geq 4$. Then, for any manifold having
 the structure of an $S^1$-bundle over the manifold $M$, we can obtain a new round fold map by a P-operation by $S^1$ on the map $f$. If $m-n \geq 2$, then the resulting
 manifold is always diffeomorphic to the product of $S^1$ and the original source manifold $M$. If $m-n=1$, then the resulting manifold
 is homeomorphic to the product of $S^{n-1}$ and a $3$-dimensional manifold having
 the structure of an $S^1$-bundle over $S^2$ and conversely, such a manifold is realized as the resulting source manifold of a round fold map obtained by
 P-operation by $S^1$ on the map $f$. If $m=n$, then the resulting manifold is homeomorphic to the
 product of $S^{n-1}$ and the $2$-dimensional torus or the Klein Bottle and conversely, these two manifolds are realized as the resulting source
 manifolds of round fold maps obtained by P-operations by $S^1$ to the original round fold map $f$. 
\subsubsection{The case where $n=3$ holds.}
Let $n=3$. In this case, for any manifold having the structure of an $S^1$-bundle over the source manifold $M$, the restriction of the bundle to the inverse image of any $2$-dimensional
 sphere smoothly embedded in the interior of the image of the map is orientable. If the bundle obtained by
 the restriction is oriented and the Euler class vanishes, then, from Theorem \ref{thm:1} or Theorem \ref{thm:2}, on the manifold having the structure
 of an $S^1$-bundle over $M$, we obtain a round fold map by a P-operation by $S^1$ to the original map $f$. \\
\ \ \ We consider the resulting source manifold of the round fold map given by a P-operation in the case where $n=3$ holds. If $m-n=m-3 \geq 2$ or $m \geq 5$ holds, then the resulting manifold
 must be diffeomorphic to the product of $S^1$ and the original source manifold. If $m-n=m-3=1$ or $m=4$ holds, then the resulting manifold must
 be diffeomorphic to the product of $S^2$ and a $3$-dimensional manifold having
 the structure of an $S^1$-bundle over $S^2$ and conversely, we obtain any such manifold as the source manifold of the resulting round fold map given by a
 P-operation by $S^1$ to the original round fold map $f$. If $m=n=3$ holds, then the resulting manifold must be diffeomorphic to the product of $S^2$ and the $2$-dimensional torus or the Klein Bottle and conversely, these manifolds
 are obtained by P-operations by $S^1$
 to the original map $f$ as the source manifolds of resulting round fold maps. 
\subsubsection{The case where $n=2$ holds.}
Let $n=2$. In this case, for a manifold having the structure of an $S^1$-bundle over the source manifold of such
 a round fold map, if the restriction of the bundle to the inverse image of any circle smoothly embedded in the
 regular value set of the map is orientable, then, by Theorem \ref{thm:3} or \ref{thm:4}, on the original manifold having the structure
 of an $S^1$-bundle, we obtain a round fold map by a P-operation by $S^1$ to the original round fold map. \\
\ \ \ We consider the resulting source manifolds given by these P-operations in the case where $n=2$ holds. If $m-n=m-2 \geq 2$ or $m \geq 4$ holds, then the resulting manifold must
 be diffeomorphic to the product of $S^1$ and the original source manifold. If $m-n=m-2=1$ or $m=3$ holds, then the resulting manifold
 must be diffeomorphic to the product of $S^1$ and a $3$-dimensional manifold having
 the structure of an $S^1$-bundle over $S^2$ and conversely, any such manifold is obtained by a
 P-operation by $S^1$ to the original round fold map. \\
\ \ \ Last, we consider the case where $(m,n)=(2,2)$ holds. On any manifold having the structure of an oriented $S^1$-bundle over the oriented torus $S^1 \times S^1$, we
 obtain a round fold map into ${\mathbb{R}}^2$ by a P-operation by $S^1$ to the original map. The resulting source manifolds are
 diffeomorphic (homeomorphic) if and only if the absolute values of the Euler numbers of the bundles coincide. Of course, for any integer $k$, we
 can obtain an oriented $S^1$-bundle whose Euler number is $k$ in this case. Note that by P-operations by $S^1$ to the original map, we obtain a family $\{M_k\}_{k \in \mathbb{Z}}$ of
 3-dimensional oriented closed and connected manifolds and a family of round fold
 maps $\{f_k:M_k \rightarrow \mathbb{R}^2\}_{k \in \mathbb{Z}}$ such that the followings hold.
\begin{enumerate}
\item $M_k$ has the structure of an oriented $S^1$-bundle over the oriented $2$-dimensional torus $S^1 \times S^1$ with Euler number $k \in \mathbb{Z}$.
\item $f_k$ is $C^{\infty}$ L-trivial.
\item The inverse image of the axis of $f_k$ is the $2$-dimensional torus $S^1 \times S^1$. 
\end{enumerate}
Furthermore, $f_k$ is not $C^{\infty}$ S-trivial for $k \neq 0$, which follows from Proposition \ref{prop:2} (\ref{prop:2.4}). \\
\ \ \ We can also obtain a round fold map from the product of $S^{n-1}$ and the Klein Bottle by a P-operation by $S^1$. 
\subsection{Further applications}
\subsubsection{Case 1}
Let $m,n \in \mathbb{N}$ and $m>n \geq 2$. Let $M$ be an $m$-dimensional closed manifold admitting
 a round fold map $f:M \rightarrow {\mathbb{R}}^n$ such that the followings hold.
\begin{enumerate}
\item $f$ is $C^{\infty}$ L-trivial.
\item The image $f(M)$ is diffeomorphic to $S^{n-1} \times [-1,1]$.
\item The inverse image of an axis of the map $f$ is homeomorphic to $S^1 \times S^{m-n}$
\item $S(f)$ consists of four connected components.
\item The regular fiber of a point in a connected component of the set ${\mathbb{R}^n}-f(S(f))$ consists of the disjoint union of two spheres.
\end{enumerate} 

We easily obtain such a map by a trivial spinning construction mentioned before Example \ref{ex:1}. The source manifold is diffeomorphic to the product of $S^{n-1}$ and a manifold homeomorphic
 to $S^1 \times S^{m-n}$. Moreover, for each connected component $C$ of $f(S(f))$, there exists a small closed tubular neighborhood $N(C)$ such that $f^{-1}(N(C))$ has the structures
 of isomorphic trivial bundles whose fibers are diffeomorphic to the standard closed disc $D^{m-n+1}$ or homeomorphic to
 the standard {\rm (}$m-n+1${\rm )}-dimensional sphere $S^{m-n+1}$ with the interior of
 a union of disjoint three {\rm (}$m-n+1${\rm )}-dimensional standard
 closed discs removed as in Definition \ref{def:2}. \\
\\
\ \ \ We apply Propositions and Theorems introduced in section \ref{sec:3}. \\
\ \ \ Let $n \geq 4$. In the case where $m-n \geq 3$ holds, any manifold having the structure
 of an orientable $S^1$-bundle over $M$ is diffeomorphic to $M \times S^1$ and there exists a manifold having the structure
 of an $S^1$-bundle over $M$ diffeomorphic to the product of $S^{n-1}$ and a manifold homeomorphic to the product of the Klein Bottle and $S^{m-n}$ too. In this
 case, on any manifold having the structure
 of an $S^1$-bundle over $M$, we obtain a round fold map by a P-operation by $S^1$ to the map $f$. In the case where $m-n=2$ holds, on
 any manifold having the structure of an orientable $S^1$-bundle over $M$ such that the restriction of the bundle to
 a regular fiber of the map $f$ with two connected components is trivial, we can construct a new round fold map by
 a P-operation by $S^1$ to the map $f$ and moreover, the bundle is a trivial $S^1$-bundle. In this case, there exists a manifold having the structure
 of an $S^1$-bundle over $M$ diffeomorphic to the product of $S^{n-1}$ and a manifold homeomorphic to the product of the Klein Bottle and $S^2$ too and on
 the manifold, we can construct a new round fold map similarly. In the case where $m-n=1$ holds, any
 manifold having the structure of an orientable $S^1$-bundle over $M$ is diffeomorphic to the product of $M$ and a manifold having the
 structure of an orientable $S^1$-bundle over the $2$-dimensional torus and on
 any manifold having the structure of an orientable $S^1$-bundle over $M$, we can construct a new round
 fold map by a P-operation by $S^1$ to the map $f$. There exists a manifold having the structure
 of an $S^1$-bundle over $M$ diffeomorphic to the product of $S^{n-1}$ and a manifold diffeomorphic to the product of the Klein Bottle and $S^1$ and on
 the manifold, we can construct a new round fold map similarly. \\
\ \ \ Let $n=3$. In the case where $m-n=m-3 \geq 3$ ($m \geq 6$) holds, on any manifold having the structure of an orientable $S^1$-bundle over $M$ such that the restriction of the bundle to
 any $2$-dimensional sphere in the set ${\mathbb{R}}^3-f(S(f))$ is trivial, we can construct a new round fold map by
 a P-operation by $S^1$ to the map $f$ and moreover, the bundle is a trivial $S^1$-bundle. There exists a manifold having the structure
 of an $S^1$-bundle over $M$ diffeomorphic to the product of $S^{n-1}=S^2$ and a manifold homeomorphic to the product of the Klein Bottle and $S^{m-n}$ and on
 the manifold, we can construct a new round fold map similarly. In the case where $m-n=m-3=2$ ($m=5$) holds, on any manifold having the structure of an orientable $S^1$-bundle over $M$ such that the restriction of the bundle to
 any $2$-dimensional sphere in the set ${\mathbb{R}}^n-f(S(f))$ and a regular fiber of the map $f$ with
 two connected components are trivial, we can construct a new round fold map by
 a P-operation by $S^1$ to the map $f$ and moreover, the bundle is a trivial $S^1$-bundle. There exists a manifold having the structure
 of an $S^1$-bundle over $M$ diffeomorphic to the product of $S^{n-1}=S^2$ and a manifold homeomorphic to the product of the Klein Bottle and $S^{m-n}=S^3$ and on
 the manifold, we can construct a new round fold map similarly. In the case where $m-n=m-3=1$ ($m=4$) holds, any manifold having the structure of an orientable $S^1$-bundle over $M$ such that the restriction of the bundle to
 any $2$-dimensional sphere in the set ${\mathbb{R}}^3-f(S(f))$ is trivial is diffeomorphic
 to the product of $M$ and a manifold having the
 structure of an orientable $S^1$-bundle over the $2$-dimensional torus and on any such manifold,
 we can construct a new round fold map by a P-operation by $S^1$ to the map $f$. There
 exists a manifold having the structure
 of an $S^1$-bundle over $M$ diffeomorphic to the product of $S^{n-1}=S^2$ and a manifold diffeomorphic to the product of the Klein Bottle and $S^1$ and on
 the manifold, we can construct a new round fold map similarly. \\ 
\ \ \ Let $n=2$. For simplicity, we only consider the case where $m-n=m-2 \geq 3$ ($m \geq 5$) holds. \\
\ \ \ Any manifold having
 the structure of an orientable $S^1$-bundle over $M$ is homeomorphic to the product of a $3$-dimensional manifold having the structure of an orientable $S^1$-bundle over the $2$-dimensional torus and $S^{m-n}$ and for
 each product of a $3$-dimensional manifold having the structure of an orientable $S^1$-bundle over the $2$-dimensional torus and $S^{m-n}=S^{m-2}$, there exists a
 manifold having the structure of an orientable $S^1$-bundle over $M$ homeomorphic to the product. Furthermore, by Proposition \ref{prop:2} (\ref{prop:2.4}), the 1st {\rm (}co{\rm )}homology group of a $3$-dimensional manifold having the structure of an
 oriented $S^1$-bundle over the $2$-dimensional torus with Euler number $k$ is isomorphic
 to $\mathbb{Z} \oplus \mathbb{Z} \oplus \mathbb{Z}/|k|\mathbb{Z}$. From the discussion here, we have the following theorem. 

\begin{Thm}
\label{thm:6}
By P-operations by $S^1$ to the map $f:M \rightarrow {\mathbb{R}}^2$ here, we have a family
 of {\rm (}$m+1${\rm )}-dimensional oriented manifolds $\{M_k\}_{k \in \mathbb{Z}}$ and
 round fold maps $\{f_k:M_k \rightarrow {\mathbb{R}}^2\}_{k \in \mathbb{Z}}$. Furthermore, the followings hold.
\begin{enumerate}
\item By any P-operation by $S^1$ to the map $f$ such that the resulting $S^1$-bundle over $M$ is orientable, the resulting source manifold of the resulting map is diffeomorphic to a manifold belonging to the family $\{M_k\}$. 
\item $M_k$ is homeomorphic to the product of a $3$-dimensional manifold having the structure of an oriented $S^1$-bundle
 over the $2$-dimensional oriented torus $S^1 \times S^1$ with Euler number $k$ and $S^{m-n}=S^{m-2}$  
\item The inverse image of an axis of $f_k$ is homeomorphic to $S^1 \times S^{m-2} \times S^1$ 
\item $f_k$ is $C^{\infty}$ L-trivial.
\item $f_k$ is not $C^{\infty}$ S-trivial if $k$ is not $0$. 
\end{enumerate}
\end{Thm}

 In this case, there exists a manifold having the structure
 of an $S^1$-bundle over $M$ diffeomorphic to the product of $S^{n-1}=S^1$ and a manifold homeomorphic to
 the product of the Klein Bottle and $S^{m-n}=S^{m-2}$ too and on this manifold, we can construct a new round
 fold map similarly. 

\subsubsection{Case 2}
Let $m,n \in \mathbb{N}$ and $m>n \geq 2$. Let $M$ be an $m$-dimensional closed manifold admitting
 a round fold map $f:M \rightarrow {\mathbb{R}}^n$ such that the followings hold.
\begin{enumerate}
\item $f$ is $C^{\infty}$ L-trivial.
\item The image $f(M)$ is diffeomorphic to $D^n$.
\item The inverse image of an axis of the map $f$ is homeomorphic to $S^1 \times S^{m-n}$ with the interior of an ($m-n+1$)-dimensional
 standard closed disc removed.
\item $S(f)$ consists of three connected components.
\item The regular fiber of a point in a connected component of the set ${\mathbb{R}^n}-f(S(f))$ consists of the disjoint union of two spheres.
\end{enumerate}

We easily obtain such a map by a trivial spinning construction mentioned before Example \ref{ex:1}. For a proper core $P$ of $f$, the manifold $f^{-1}({\mathbb{R}}^n-{\rm Int} P)$ is diffeomorphic
 to the product of $S^{n-1}$ and a manifold homeomorphic
 to $S^1 \times S^{m-n}$ with the interior of an ($m-n+1$)-dimensional standard closed disc removed. Moreover, for each connected component $C$ of $f(S(f))$, there exists a small closed tubular neighborhood $N(C)$ such that $f^{-1}(N(C))$ has the structures
 of isomorphic trivial bundles whose fibers are diffeomorphic to the standard closed disc $D^{m-n+1}$ or homeomorphic to
 the standard {\rm (}$m-n+1${\rm )}-dimensional sphere $S^{m-n+1}$ with the interior of
 a union of disjoint three {\rm (}$m-n+1${\rm )}-dimensional standard
 closed discs removed as in Definition \ref{def:2}. \\
\ \ \ We consider the case where $m-n \geq 3$ holds for simplicity. \\
\ \ \ Let $n \geq 3$ hold. In this case, the 2nd homology groups $H_2(M;\mathbb{Z})$ and $H_2(f^{-1}({\mathbb{R}}^n-{\rm Int} P);\mathbb{Z})$ and the
 2nd cohomology groups $H^2(M;\mathbb{Z})$ and $H^2(f^{-1}({\mathbb{R}}^n-{\rm Int} P);\mathbb{Z})$ are zero and any orientable $S^1$-bundle over $M$ is trivial. We can construct a round fold map
 from $M \times S^1$ into ${\mathbb{R}}^n$ by a P-operation by $S^1$ to $f$. In this case, according to a discussion
 similar to that of Case 1, there exists a manifold having the
 structure of an $S^1$-bundle over $M$ which is not orientable admitting a round fold map into ${\mathbb{R}}^n$ obtained by a P-operation by $S^1$ to $f$. \\  
\ \ \ Let $n=2$ hold. In this case, the 2nd homology groups $H_2(M;\mathbb{Z})$ and $H_2(f^{-1}({\mathbb{R}}^n-{\rm Int} P);\mathbb{Z})$ and the
 2nd cohomology groups $H^2(M;\mathbb{Z})$ and $H^2(f^{-1}({\mathbb{R}}^n-{\rm Int} P);\mathbb{Z})$ are all $\mathbb{Z}$. Moreover, the
 homomorphisms between the 2nd
 homology and cohomology groups induced by the inclusions of $f^{-1}({\mathbb{R}}^n-{\rm Int} P)$ into $M$ are
 isomorphisms. Let $\nu \in H^2(M;\mathbb{Z}) \cong \mathbb{Z}$ be a generator of the group
 and ${\nu}^{\prime} \in H_2(f^{-1}({\mathbb{R}}^n-{\rm Int} P);\mathbb{Z})$ be the
 generator of the group defined as the inverse image of $\nu$ of the isomorphism induced by the natural inclusion. Let $M_k$ be an oriented
 manifold having the structure of an oriented $S^1$-bundle over $M$ whose Euler class is $k\nu$. The
 restriction of the bundle $M_k$ to $f^{-1}({\mathbb{R}}^n-{\rm Int} P)$ is an oriented $S^1$-bundle whose Euler
 class is $k {\nu}^{\prime}$. Since we
 have $H_1(f^{-1}({\mathbb{R}}^n-{\rm Int} P);\mathbb{Z}) \cong \mathbb{Z} \oplus \mathbb{Z}$, by Proposition \ref{prop:2} and
 the homeomorphism type of the manifold $f^{-1}({\mathbb{R}}^n-{\rm Int} P)$, the 1st homology group of
 the total space of this obtained bundle whose coefficient
 ring is $\mathbb{Z}$ is isomorphic to $\mathbb{Z} \oplus \mathbb{Z} \oplus (\mathbb{Z}/|k| \mathbb{Z})$. Thus, the 1st homology
 group $H_1(M_k;\mathbb{Z})$ is isomorphic to $\mathbb{Z} \oplus (\mathbb{Z}/|k| \mathbb{Z})$. From this observation, we may apply some
 of our Propositions and Thorems in section \ref{sec:3} and have the following theorem.  

\begin{Thm}
\label{thm:7}
By P-operations by $S^1$ to the map $f:M \rightarrow {\mathbb{R}}^2$ here, we have a family
 of {\rm (}$m+1${\rm )}-dimensional oriented manifolds $\{M_k\}_{k \in \mathbb{Z}}$ and
 round fold maps $\{f_k:M_k \rightarrow {\mathbb{R}}^2\}_{k \in \mathbb{Z}}$. Furthermore, we have the followings.
\begin{enumerate}
\item By any P-operation by $S^1$ to the map $f$ such that the resulting $S^1$-bundle over $M$ is orientable, the resulting source manifold of the resulting map is diffeomorphic to a manifold belonging to the family $\{M_k\}$.
\item $M_k$ has the structure of an oriented $S^1$-bundle over $M$ such that the 1st homology group of the total space of the bundle
 obtained by the restriction of the bundle $M_k$ to $f^{-1}({\mathbb{R}}^n-{\rm Int} P)$ with coefficient ring $\mathbb{Z}$ is
 isomorphic to $\mathbb{Z} \oplus \mathbb{Z} \oplus (\mathbb{Z}/|k| \mathbb{Z})$ and that the 1st homology
 group $H_1(M_k;\mathbb{Z})$ is isomorphic to $\mathbb{Z} \oplus (\mathbb{Z}/|k| \mathbb{Z})$.
\item The inverse image of an axis of $f_k$ is diffeomorphic to the product of the inverse image of an axis of $f$ and $S^1$. 
\item $f_k$ is $C^{\infty}$ L-trivial.
\item $f_k$ is not $C^{\infty}$ S-trivial if $k$ is not $0$. 
\end{enumerate}
\end{Thm} 

\subsubsection{Case 3}

Let $m,n \in \mathbb{N}$ and $m>n \geq 2$. We also assume that $m \geq 2n$ and $m-n \geq 3$ hold. Let $M_i$ be a closed and connected oriented
 manifold admitting a round fold map $f_i:M_i \rightarrow {\mathbb{R}}^n$ satisfying the conditions similar to those mentioned
 in the beginning of Case 2 ($i=1,2$). We also assume that the fiber of a proper core of $f_1$ is diffeomorphic to the ($m-n$)-dimensional standard sphere
 $S^{m-n}$ and that $f_2$ is obtained by a trivial spinning construction. Let us construct a round fold map from the connected sum $M$ of the two manifolds
 $M_1$ and $M_2$ into ${\mathbb{R}}^n$ by applying a method used in \cite{kitazawa}, \cite{kitazawa 2} and \cite{kitazawa 4}. \\
\ \ \ Let $P_i$ be a proper core of $f_i$ ($i=1,2$) and let $Q$ be a small closed tubular neighborhood of the connected component of $f_2(S(f_2))$ which is the
 boundary of $f_2(M_2)$. It easily follows that ${f_1} {\mid}_{{f_1}^{-1}(P_1)}:{f_1}^{-1}(P_1) \rightarrow P_1$ gives the structure of a trivial $S^{m-n}$-bundle over
 $P_1$ and that ${f_2}^{-1}(Q)$ is a closed tubular neighborhood of ${f_2}^{-1}(C) \subset M_2$. \\
\ \ \ Since the map $f_2$ is obtained by a trivial spinning construction, the inclusion of ${f_2}^{-1}(C)$ into $M_2$ is homotopic
 to a section of the trivial bundle over the proper core $P_2$ of $f_2$ above given by $f_2 {\mid}_{{f_2}^{-1}(\partial P_2)}:{f_2}^{-1}(\partial P_2) \rightarrow \partial P_2$
 and the section is null-homotopic in the manifolds ${f_2}^{-1}(P_2)$ and $M_2$ since the fiber of
 this trivial bundle is a sphere of dimension $m-n>n-1$ and ($n-1$)-connected. From
 the inequality $m \geq 2n= 2(n-1)+2$, the inclusion of $f^{-1}(C)$ into $M_2$ is unknot in the $C^{\infty}$ category and ${f_2}^{-1}(Q)$ has the structure of a trivial $D^{m-n+1}$-bundle over an ($n-1$)-dimensional sphere $\partial Q \bigcap f_2(M)$. More
 precisely, $f_2 {\mid}_{\partial {f_2}^{-1}(Q)}:\partial {f_2}^{-1}(Q) \rightarrow \partial Q \bigcap f_2(M)$ gives
 the structure of a subbundle of the bundle. \\
\ \ \ Let $V_1:={f_1}^{-1}(P_1)$ and $V_2:={f_2}^{-1}(Q)$. From the discussion above, we may regard that the following holds for a diffeomorphism $\Phi:\partial V_2 \rightarrow \partial V_1$ regarded as a bundle
 isomorphism between the two trivial $S^{m-n+1}$-bundles over $S^{n-1}$ inducing a diffeomorphism between the base spaces and an orientation reversing diffeomorphism
 $\Psi:\partial D^m \rightarrow \partial D^m$ extending to a diffeomorphism on $D^m$ or from $M_2-(M_2-D^m)$
 onto $M_1-(M_1-D^m)$, where for two manifolds $X_1$ and $X_2$, $X_1 \cong X_2$ means
 that $X_1$ and $X_2$ are diffeomorphic. 

\begin{eqnarray*}
& & (M_1-{\rm Int} V_1) {\bigcup}_{\Phi} (M_2-{\rm Int} V_2) \\
& \cong & (M_1-{\rm Int} V_1) {\bigcup}_{\Phi} ((D^m-{\rm Int} V_2) \bigcup (M_2-{\rm Int} D^m)) \\
& \cong & (M_1-{\rm Int} V_1) {\bigcup}_{\Phi} ((S^m-({\rm Int} V_2 \sqcup {\rm Int} D^m)) {\bigcup}_{\Psi} (M_2-{\rm Int} D^m)) \\
& \cong & (M_1-{\rm Int} D^m) {\bigcup}_{\Psi} (M_2-{\rm Int} D^m)
\end{eqnarray*}

\ \ \ This means that the resulting manifold $M$, which is the connected sum of $M_1$ and $M_2$, admits a round fold map into ${\mathbb{R}}^n$. More precisely, the
 resulting map
 is obtained by gluing the two maps ${f_1} {\mid}_{M_1-{\rm Int} V_1}$ and ${f_2} {\mid}_{M_2-{\rm Int} V_2}$. The resulting map $f:M \rightarrow {\mathbb{R}}^n$
 is $C^{\infty}$ L-trivial and the inverse image of the axis of $f$ is represented as a manifold diffeomorphic to the connected sum of two
 manifolds homeomorphic to $S^1 \times S^{n-1}$ with the interior of an ($m-n+1$)-dimensional standard closed disc removed. \\
\ \ \ The cohomology group $H^2(M;\mathbb{Z})$ is zero ($n \neq 2$) or $\mathbb{Z} \oplus \mathbb{Z}$ ($n=2$). The manifolds ${f_1}^{-1}({\mathbb{R}}^n-{\rm Int} P_1)$
 and ${f_2}^{-1}({\mathbb{R}}^n-(P_2 \bigcup Q))$ are diffeomorphic to products of $S^{n-1}$ and manifolds homeomorphic to $S^1 \times S^{m-n}$ with the
 interior of an ($m-n+1$)-dimensional standard closed disc removed. For a proper core $P$ of $f$, ${f_1}^{-1}({\mathbb{R}}^n-{\rm Int} P_1) \bigcup {f_2}^{-1}({\mathbb{R}}^n-(P_2 \bigcup Q))$
 is regarded as the inverse image $f^{-1}({\mathbb{R}}^n-{\rm Int} P)$. In the case where $n \geq 3$ is assumed, any orientable $S^1$-bundle over $M$ is
 trivial since the cohomology group $H^2(M;\mathbb{Z})$ is zero and we can construct a round fold map from $M \times S^1$ into ${\mathbb{R}}^n$ by a P-operation by $S^1$.
In the case where $n=2$ is assumed, we
 have $H^2(f^{-1}({\mathbb{R}}^2-{\rm Int} P);\mathbb{Z}) \cong \mathbb{Z} \oplus \mathbb{Z}$ and the homomorphism between the 2nd cohomology
 groups induced by the natural inclusion of $f^{-1}({\mathbb{R}}^2-{\rm Int} P)$ into $M$ is an isomorphism. In
 this case, we can set ${\nu}_1$ as a generator of the group $H^2({f_1}^{-1}({\mathbb{R}}^n-{\rm Int} P_1);\mathbb{Z}) \cong \mathbb{Z}$ and ${\nu}_2$ as a generator of the group $H^2({f_2}^{-1}({\mathbb{R}}^n-(P_2 \bigcup Q));\mathbb{Z}) \cong \mathbb{Z}$
 and by regarding $H^2({f_1}^{-1}({\mathbb{R}}^n-{\rm Int} P_1);\mathbb{Z}) \oplus H^2({f_2}^{-1}({\mathbb{R}}^n-(P_2 \bigcup Q));\mathbb{Z})$ as $H^2(f^{-1}({\mathbb{R}}^2-{\rm Int} P);\mathbb{Z})$ and using the isomorphism between the cohomology groups before, we can define the class ${{\nu}_i}^{\prime} \in H^2(M;\mathbb{Z})$ corresponding to ${\nu}_i$ ($i=1,2$). There exists an
 oriented manifold $M_{(k_1,k_2)}$ having the structure of an oriented $S^1$-bundle over $M$ whose Euler class is $k_1 {{\nu}_1}^{\prime}+k_2 {{\nu}_2}^{\prime}$
 admitting a round fold map into ${\mathbb{R}}^2$ obtained by a P-operation by $S^1$ to $f$. Thus, we
 have the following theorem. 

\begin{Thm}
\label{thm:8}
By P-operations by $S^1$ to the map $f:M \rightarrow {\mathbb{R}}^2$ here, we have a family
 of {\rm (}$m+1${\rm )}-dimensional oriented manifolds $\{M_{(k_1,k_2)}\}_{(k_1,k_2) \in \mathbb{Z} \oplus \mathbb{Z}}$ and
 round fold maps $\{f_{(k_1,k_2)}:M_{(k_1,k_2)} \rightarrow {\mathbb{R}}^2\}_{(k_1,k_2) \in \mathbb{Z} \oplus \mathbb{Z}}$. Furthermore, we have the followings.
\begin{enumerate}
\item By any P-operation by $S^1$ to the map $f$ such that the resulting $S^1$-bundle over $M$ is orientable, the resulting source manifold of the resulting map is diffeomorphic to a manifold belonging to the family $\{M_{(k_1,k_2)}\}$.
\item $M_{(k_1.k_2)}$ has the structure of an oriented $S^1$-bundle over $M$ and the followings hold.
\begin{enumerate}
\item The 1st homology group of the total space of the bundle
 obtained by the restriction of the bundle $M_{(k_1,k_2)}$ to ${f_1}^{-1}({\mathbb{R}}^2-{\rm Int} P_1) \subset M$ with coefficient ring $\mathbb{Z}$ is
 isomorphic to $\mathbb{Z} \oplus \mathbb{Z} \oplus (\mathbb{Z}/|k_1| \mathbb{Z})$.
\item The 1st homology
 group of the total space of the bundle obtained by the restriction of the bundle $M_{(k_1,k_2)}$ to ${f_2}^{-1}({\mathbb{R}}^n-(P_2 \bigcup Q)) \subset M$ is isomorphic to $\mathbb{Z} \oplus \mathbb{Z} \oplus (\mathbb{Z}/|k_2| \mathbb{Z})$.
\item Let $k_1 \neq 0$ and $k_2 \neq 0$ hold and let $k_0$ be the LCM of $|k_1|$ and $|k_2|$. Then the 1st homology group of $M_{(k_1,k_2)}$ with coefficient ring $\mathbb{Z}$ is isomorphic to
 $\mathbb{Z} \oplus \mathbb{Z} \oplus (\mathbb{Z}/k_0 \mathbb{Z})$. 
\item Let $k_1=0$ or $k_2=0$ hold. Then the 1st homology group of $M_{(k_1,k_2)}$ with coefficient ring $\mathbb{Z}$ is isomorphic to
 $\mathbb{Z} \oplus \mathbb{Z}$.
\end{enumerate}
\item The inverse image of an axis of $f_{(k_1,k_2)}$ is diffeomorphic to the product of the inverse image of an axis of $f$ and $S^1$. 
\item $f_{(k_1,k_2)}$ is $C^{\infty}$ L-trivial.
\item If $k_1 \neq 0$ and $k_2 \neq 0$ hold, then $f_{(k_1,k_2)}$ is not $C^{\infty}$ S-trivial.
\item $f_{(k_1,k_2)}$ and $f_{(k_3,k_4)}$ are not $C^{\infty}$ equivalent if $|k_1| \neq |k_3|$ or $|k_2| \neq |k_4|$ holds. 
\end{enumerate}
\end{Thm}  



\end{document}